\documentclass[11pt,reqno]{amsart}

\usepackage{graphicx}
\usepackage{cite}
\usepackage{xspace}
\usepackage{amsmath,amsthm,amsfonts,amssymb,latexsym,mathrsfs,color}
\usepackage{hyperref}

\newtheorem{theorem}{Theorem}%[section]
\newtheorem{corollary}[theorem]{Corollary}
\newtheorem{proposition}[theorem]{Proposition}

\newtheorem{lemma}[theorem]{Lemma}
\newtheorem{definition}[theorem]{Definition}
\newtheorem{example}[theorem]{Example}

\newcommand{\csum}{{\rm csum\,}}

\newcommand{\negg}{{\rm neg\,}}

\newcommand{\cpkk}{{\rm cpk\,}}
\newcommand{\cdasc}{{\rm cdasc\,}}
\newcommand{\cddes}{{\rm cddes\,}}

\newcommand{\crun}{{\rm crun\,}}

\newcommand{\Asc}{{\rm Asc\,}}

\newcommand{\cda}{{\rm cda\,}}
\newcommand{\CDD}{{\rm CDD\,}}

\newcommand{\ap}{{\rm ap\,}}
\newcommand{\lap}{{\rm lap\,}}
\newcommand{\lpk}{{\rm lpk\,}}
\newcommand{\single}{{\rm single\,}}

\newcommand{\des}{{\rm des\,}}

\newcommand{\exc}{{\rm exc\,}}
\newcommand{\wexc}{{\rm wexc\,}}

\newcommand{\fexc}{{\rm fexc\,}}
\newcommand{\aexc}{{\rm aexc\,}}
\newcommand{\drop}{{\rm drop\,}}
\newcommand{\we}{{\rm wexc\,}}

\newcommand{\cyc}{{\rm cyc\,}}

\newcommand{\fix}{{\rm fix\,}}

\newcommand{\mcc}{{\mathcal C}}
\newcommand{\cc}{{\mathcal C}}

\newcommand{\md}{\mathcal{D}}
\newcommand{\mdn}{\mathcal{D}_n}

\newcommand{\D}{\mathfrak{D}}

\newcommand{\msn}{\mathcal{S}_n}

\newcommand{\ms}{\mathcal{S}}

\newcommand{\lrf}[1]{\lfloor #1\rfloor}

\newcommand{\mbn}{{\mathcal S}^B_n}
\newcommand{\mb}{{\mathcal S}^B}

\newcommand{\mq}{\mathcal{Q}}
\newcommand{\mqn}{\mathcal{Q}_n}
\newcommand{\z}{ \mathbb{Z}}

\title{Excedance-type polynomials and gamma-positivity}
\author[S.-M.~Ma]{Shi-Mei Ma}
\address{School of Mathematics and Statistics,
        Northeastern University at Qinhuangdao,
         Hebei 066000, P.R. China}
\email{shimeimapapers@163.com (S.-M. Ma)}
\author[J.~Ma]{Jun Ma}
\address{Department of mathematics, Shanghai Jiao Tong University, Shanghai, P.R. China}
\email{majun904@sjtu.edu.cn(J.~Ma)}
\author{Jean Yeh}
\address{Department of Mathematics, National Kaohsiung Normal University, Kaohsiung 82444, Taiwan}
\email{chunchenyeh@nknu.edu.tw}
\author[Y.-N. Yeh]{Yeong-Nan Yeh}
\address{Institute of Mathematics,
        Academia Sinica, Taipei, Taiwan}
\email{mayeh@math.sinica.edu.tw (Y.-N. Yeh)}
\begin{document}

\maketitle
\begin{abstract}
The object of this paper is to give a systematic treatment of excedance-type polynomials.
We first give a sufficient condition for a sequence of polynomials to have
alternatingly increasing property, and then we present a systematic study of the joint distribution of
excedances, fixed points and cycles of permutations and derangements, signed or not, colored or not.
Let $p\in [0,1]$ and $q\in [0,1]$ be two given real numbers.
We prove that the $\cyc$ $q$-Eulerian polynomials of permutations are bi-$\gamma$-positive, and
the $\fix$ and $\cyc$ $(p,q)$-Eulerian polynomials of permutations are alternatingly increasing, and so they are unimodal with modes in the middle, where $\fix$ and $\cyc$ are the fixed point and cycle statistics. When $p=1$ and $q=1/2$,
we find a combinatorial interpretation of the bi-$\gamma$-coefficients of the $(p,q)$-Eulerian polynomials.
We then study excedance and flag excedance statistics of signed permutations and colored permutations. In particular, we establish the relationships
between the $(p,q)$-Eulerian polynomials and some multivariate Eulerian polynomials.
Our results unify and generalize a variety of recent results.
\bigskip

\end{abstract}
\date{\today}
%\date{\today}
%%%%%%%%%%%%%%%%%%%%%%%%%%%%%%%%%%%%%%%%%%%%%%%%%%%%%%%%%%%%%%%%%%%%%%%%%%
%%%%%%%%%%%%%%%%%%%%%%%%%%%%%%%%%%%%%%%%%%%%%%%%%%%%%%%%%%%%%%%%%%%%%%%%%%
%%%%%%%%%%%%%%%%%%%%%%%%%%%%%%%%%%%%%%%%%%%%%%%%%%%%%%%%%%%%%%%%%%%%%%%%%%
%%%%%%%%%%%%%%%%%%%%%%%%%%%%%%%%%%%%%%%%%%%%%%%%%%%%%%%%%%%%%%%%%%%%%%%%%%
%%%%%%%%%%%%%%%%%%%%%%%%%%%%%%%%%%%%%%%%%%%%%%%%%%%%%%%%%%%%%%%%%%%%%%%%%%
%%%%%%%%%%%%%%%%%%%%%%%%%%%%%%%%%%%%%%%%%%%%%%%%%%%%%%%%%%%%%%%%%%%%%%%%%%
%%%%%%%%%%%%%%%%%%%%%%%%%%%%%%%%%%%%%%%%%%%
\section{Introduction}
%%%%%%%%%%%%%%%%%%%%%%%%%%%%%%%%%%%%%%%%%%%
%%%%%%%%%%%%%%%%%%%%%%%%%%%%%%%%%%%%%%%%%%%%%%%%%%%%%%%%%%%%%%%%%%%%%%%%%%%%%%%%%
%%%%%%%%%%%%%%%%%%%%%%%%%%%%%%%%%%%%%%%%%%%%%%%%%%%%%%%%%%%%%%%%%%%%%%%%%%%%%%%%%
%%%%%%%%%%%%%%%%%%%%%%%%%%%%%%%%%%%%%%%%%%%%%%%%%
%%%%%%%%%%%%%%%%%%%%%%%%%%%5%%%%
%%%%%%%%%%%%%%%%%%%%%%%%%%%%%%%%%
%%%%%%%%%%%%%%%%%%%%%%%%%%%%%%%%%%%%%%%%%%%%%%%%%%%%
Motivated by the recent papers~\cite{Athanasiadis20,Beck2019,Branden18} and~\cite{Solus19},
this paper is concerned with unimodal polynomials with the modes in the middle.
The aim of this paper is twofold. The first is to develop techniques for the unimodality problems. The second is to
give a systematic study of the joint distribution of
excedance, fixed point and cycle statistics of permutations and derangements, signed or not, colored or not.
Our results unify and generalize some recent results of Athanasiadis~\cite{Athanasiadis14,Athanasiadis20}, Bagno and Garber~\cite{Bagno04}, Br\"and\'{e}n~\cite{Branden08,Branden18},
Chen et al.~\cite{Chen09}, Chow~\cite{Chow08,Chow09}, Chow and Mansour~\cite{Chow10}, Foata and Han~\cite{Foata11}, Juhnke-Kubitzke et al.~\cite{Sieg19},
Mongelli~\cite{Mongelli13}, Petersen~\cite{Petersen07}, Shin and Zeng~\cite{Zeng12,Zeng16}.

Let $f(x)=\sum_{i=0}^nf_ix^i$ be a polynomial with nonnegative coefficients.
We say that $f(x)$ is {\it unimodal} if $f_0\leqslant f_1\leqslant \cdots\leqslant f_k\geqslant f_{k+1}\geqslant\cdots \geqslant f_n$ for some $k$, where the index $k$ is called the {\it mode} of $f(x)$. The polynomials $f(x)$ is said to be {\it spiral} if
$$f_n\leqslant f_0\leqslant f_{n-1}\leqslant f_1\leqslant \cdots\leqslant f_{\lrf{n/2}}.$$
If $f(x)$ is symmetric with the centre of symmetry $\lrf{n/2}$, i.e., $f_i=f_{n-i}$ for all indices $0\leqslant i\leqslant n$,
then it can be expanded as
$$f(x)=\sum_{k=0}^{\lrf{{n}/{2}}}\gamma_kx^k(1+x)^{n-2k}.$$
Following Gal~\cite{Gal05}, the polynomial $f(x)$ is said to be {\it $\gamma$-positive}
if $\gamma_k\geqslant 0$ for all $0\leqslant k\leqslant \lrf{{n}/{2}}$, and the coefficients $\gamma_k$ are called the $\gamma$-coefficients of $f(x)$.
Following~\cite[Definition 2.9]{Schepers13}, the polynomial $f(x)$ is {\it alternatingly increasing} if
$$f_0\leqslant f_n\leqslant f_1\leqslant f_{n-1}\leqslant\cdots \leqslant f_{\lrf{{(n+1)}/{2}}}.$$
It is clear that if $f(x)$ is spiral and $\deg f(x)=n$, then $x^nf(1/x)$ is alternatingly increasing, and vice versa.
Both $\gamma$-positivity and alternatingly increasing property imply unimodality.
In the past decades, unimodal polynomials arise often in combinatorics and geometry,
see~\cite{Athanasiadis17,Lin15} and references therein.
It should be noted that the definition of
alternatingly increasing property first appeared in the work of Beck and Stapledon~\cite{Beck2010}. Very recently, Beck, Jochemko and McCullough~\cite{Beck2019} and Solus~\cite{Solus19} studied the
alternatingly increasing property of several $h^*$-polynomials as well as several refined Eulerian polynomials.
This paper is motivated by empirical evidence which suggests that
some multivariate Eulerian polynomials are unimodal with modes in the middle.
%As pointed out by Brenti~\cite{Brenti9401},
%to prove the unimodality of a polynomial can sometimes be a very difficult task requiring the use of intricate combinatorial constructions or of refined mathematical tools.
%, and study the $\gamma$-positivity and bi-$\gamma$-positivity of some polynomials, which implies that these polynomials are unimodal with modes in the middle.
%In the following,
%we collect
%several definitions, notation, and results that will be used in the sequel.

We first recall an elementary result.
\begin{proposition}[{\cite{Beck2010,Branden18}}]\label{prop01}
Let $f(x)$ be a polynomial of degree $n$.
There is a unique symmetric decomposition $f(x)= a(x)+xb(x)$, where \begin{equation}\label{ax-bx-prop01}
a(x)=\frac{f(x)-x^{n+1}f(1/x)}{1-x},~b(x)=\frac{x^nf(1/x)-f(x)}{1-x}.
\end{equation}
\end{proposition}
By using~\eqref{ax-bx-prop01}, it is easy to verify that if $f(0)\neq0$, then $\deg a(x)=n$ and $\deg b(x)\leqslant n-1$.
%if $f(0)=0$, then $\deg a(x)\leqslant n-1$ and $\deg b(x)=n-1$.
We call the ordered pair of polynomials $(a(x),b(x))$ the {\it symmetric decomposition} of $f(x)$, since $a(x)$ and $b(x)$ are both symmetric.

\begin{definition}\label{def01}
Let $(a(x),b(x))$ be the {\it symmetric decomposition} of a polynomial $f(x)$.
If $a(x)$ and $b(x)$ are both $\gamma$-positive, then $f(x)$ is said to be {\it bi-$\gamma$-positive}.
The $\gamma$-coefficients of $a(x)$ and $b(x)$ are called the bi-$\gamma$-coefficients of $f(x)$.
\end{definition}
%\begin{definition}\label{def01}
%Let $f(x)=\sum_{i=0}^nf_ix^i$. Assume that $f(x)=$ can be expanded as
%$$f(x)=\sum_{i=0}^{\lrf{{n}/{2}}}\xi_ix^i(1+x)^{n-2i}+\sum_{j=1}^{\lrf{{(n+1)}/{2}}}\eta_jx^j(1+x)^{n+1-2j}.$$
%Then $f(x)$ is said to be {\it bi-$\gamma$-positive}
%if $\xi_i\geqslant 0$ for all $0\leqslant i\leqslant \lrf{\frac{n}{2}}$ and $\eta_j\geqslant 0$ for all $1\leqslant j\leqslant \lrf{\frac{n+1}{2}}$.
%\end{definition}
As pointed out by Br\"and\'en and Solus~\cite{Branden18},
the polynomial $f(x)$ is alternatingly increasing if and only if the pair of polynomials in its symmetric decomposition are both unimodal
and have only nonnegative coefficients. Therefore, bi-$\gamma$-positivity implies alternatingly increasing property.
If $f(x)$ is $\gamma$-positive, then $f(x)$ is also bi-$\gamma$-positive but not vice versa. We now provide a connection between
$\gamma$-positivity and bi-$\gamma$-positivity.
\begin{proposition}
If $f(x)$ is $\gamma$-positive and $f(0)=0$, then $f'(x)$ is bi-$\gamma$-positive.
\end{proposition}
\begin{proof}
Assume that $f(x)=\sum_{k=1}^{\lrf{{n}/{2}}}\gamma_kx^k(1+x)^{n-2k}$, where $\gamma_k\geqslant0$ for all $1\leqslant k\leqslant \lrf{n/2}$.
Then we have
\begin{align*}
f'(x)&=\sum_{k=1}^{\lrf{{n}/{2}}}k\gamma_kx^{k-1}(1+x)^{n-2k}+\sum_{k=1}^{\lrf{{(n-1)}/{2}}}(n-2k)\gamma_kx^{k}(1+x)^{n-2k-1}\\
%&=\sum_{i=0}^{\lrf{{(n-2)}/{2}}}(i+1)\gamma_{i+1}x^{i}(1+x)^{n-2i-2}+\sum_{k=1}^{\lrf{{(n-1)}/{2}}}(n-2k)\gamma_kx^{k}(1+x)^{n-2k-1}\\
&=\sum_{i=0}^{\lrf{{(n-2)}/{2}}}(i+1)\gamma_{i+1}x^{i}(1+x)^{n-2i-2}+x\sum_{j=0}^{\lrf{{(n-3)}/{2}}}(n-2j-2)\gamma_{j+1}x^{j}(1+x)^{n-2j-3}.
\end{align*}
Therefore, $f'(x)$ is bi-$\gamma$-positive.
\end{proof}

The following simple result will be used
repeatedly in our discussion.
\begin{lemma}\label{lemma01}
Let $f(x)=\sum_{i=0}^nf_ix^i$ and $g(x)=\sum_{j=0}^mg_jx^j$.
If $f(x)$ is $\gamma$-positive and $g(x)$ is bi-$\gamma$-positive, then $f(x)g(x)$ is bi-$\gamma$-positive. In particular,
the product of two $\gamma$-positive polynomials is also $\gamma$-positive.
\end{lemma}
\begin{proof}
Assume that $f(x)=\sum_{k=0}^{\lrf{{n}/{2}}}\gamma_kx^k(1+x)^{n-2k}$ and
\begin{align*}
g(x)&=\sum_{i=0}^{\lrf{{m}/{2}}}\xi_ix^i(1+x)^{m-2i}+\sum_{j=1}^{\lrf{{(m+1)}/{2}}}\eta_jx^j(1+x)^{m+1-2j},
\end{align*}
where $\gamma_k,\xi_i$ and $\eta_j$ are all nonnegative numbers.
Then $f(x)g(x)$ can be expanded as
$$f(x)g(x)=\sum_{s=0}^{\lrf{(n+m)/2}}\alpha_sx^{s}(1+x)^{n+m-2s}+\sum_{t=1}^{\lrf{(n+m+1)/2}}\beta_{t}x^{t}(1+x)^{n+m+1-2t},$$
where $\alpha_s=\sum_{k+i=s}\gamma_k\xi_i$ and $\beta_{t}=\sum_{k+j=t}\gamma_k\eta_j$, which yields the desired result.
\end{proof}

We now recall a definition.
\begin{definition}[{\cite[Definition~4]{Ma19}}]
Let $p(x,y)$ be a bivariate polynomial. Suppose $p(x,y)$ can be expanded as
\begin{equation}\label{defpartial}
p(x,y)=\sum_{i=0}^ny^i\sum_{j=0}^{\lrf{(n-i)/2}}\mu_{n,i,j}x^j(1+x)^{n-i-2j}.
\end{equation}
If $\mu_{n,i,j}\geqslant 0$ for all $0\leqslant i\leqslant n$ and $0\leqslant j\leqslant \lrf{(n-i)/2}$, then we say that $p(x,y)$ is a partial $\gamma$-positive polynomial.
The numbers $\mu_{n,i,j}$ are called the partial $\gamma$-coefficients of $p(x,y)$.
\end{definition}

It should be noted that partial $\gamma$-positive polynomials frequently appear in combinatorics and geometry, see~\cite{Athanasiadis17,Hwang20,Zeng2020,Ma19} for instance.
We can now conclude the first main result of this paper.
\begin{theorem}\label{bigammpartial}
Suppose the polynomial $p(x,y)$ has the expression~\eqref{defpartial} and $\deg p(x,1)=n-1$, where $n$ is a positive integer.
If $p(x,y)$ is partial $\gamma$-positive, $p(x,1)$ is bi-$\gamma$-positive and $0\leqslant y\leqslant 1$ is a given real number, then
$p(x,y)$ is alternatingly increasing.
\end{theorem}

Let $[n]=\{1,2,\ldots,n\}$.
Let $\msn$ be the set of all permutations of $[n]$ and
let $\pi=\pi(1)\pi(2)\cdots\pi(n)\in\msn$. We say that $i$ is a {\it descent} (resp.~{\it excedance}, {\it drop}, {\it fixed point})
if $\pi(i)>\pi(i+1)$ (resp.~$\pi(i)>i$,~$\pi(i)<i$,~$\pi(i)=i$).
Let $\des(\pi)$, $\exc(\pi)$, $\drop(\pi)$, $\fix(\pi)$ and $\cyc(\pi)$ be the number of descents, excedances, drops, fixed points and cycles of $\pi$, respectively.
It is well known that descents, excedances and drops are equidistributed over $\msn$, and their common enumerative polynomial is the classical Eulerian polynomial:
$$A_n(x)=\sum_{\pi\in\msn}x^{\des(\pi)}=\sum_{\pi\in\msn}x^{\exc(\pi)}=\sum_{\pi\in\msn}x^{\drop(\pi)}.$$

An index $i\in[n]$ is called a {\it double descent} of $\pi\in\msn$ if $\pi(i-1)>\pi(i)>\pi(i+1)$, where $\pi(0)=\pi(n+1)=0$.
Foata and Sch\"utzenberger~\cite{Foata70} found the following notable result.
\begin{proposition}[\cite{Foata70}]\label{Foata70}
One has
$$A_n(x)=\sum_{i=0}^{\lrf{(n-1)/2}}\gamma_{n,i}x^i(1+x)^{n-1-2i},$$
where $\gamma_{n,i}$ is the number of permutations $\pi\in \msn$ which have no double descents and $\des(\pi)=i$.
\end{proposition}

An element $\pi\in\msn$ is called a {\it derangement} if $\fix(\pi)=0$. Let $\mdn$ be the set of all derangements in $\msn$.
The {\it derangement polynomials} are defined by $d_n(x)=\sum_{\pi\in\mdn}x^{\exc(\pi)}$.
It is well known that the generating function of $d_n(x)$ is given as follows (see~\cite[Proposition~6]{Brenti90}):
\begin{equation}\label{dxz-EGF}
d(x;z)=\sum_{n=0}^{\infty}d_n(x)\frac{z^n}{n!}=\frac{1-x}{e^{xz}-xe^{z}}.
\end{equation}
The cardinality of a set $A$ will be
denoted by $\#A$.
Let $\cda(\pi)=\#\{i:\pi^{-1}(i)<i<\pi(i)\}$ be the number of {\it cycle double ascents} of $\pi$.
By using the theory of continued fractions, Shin and Zeng~\cite[Theorem~11]{Zeng12} obtained the following result.
\begin{proposition}[\cite{Zeng12}]\label{Zeng12}
Let $\md_{n,k}=\{\pi\in \msn: \fix(\pi)=0,~\cda(\pi)=0,~\exc(\pi)=k\}$. Then
\begin{equation}\label{dnxq-def}
d_n(x,q)=\sum_{\pi\in\mdn}x^{\exc(\pi)}q^{\cyc(\pi)}=\sum_{k=1}^{\lrf{n/2}}\sum_{\pi\in \md_{n,k}}q^{\cyc(\pi)}x^{k}(1+x)^{n-2k}.
\end{equation}
\end{proposition}

Let $\pm[n]=[n]\cup\{\overline{1},\ldots,\overline{n}\}$, where $\overline{i}=-i$.
Let $\mbn$ be the {\it hyperoctahedral group} of rank $n$.
Elements of $\mbn$ are permutations of $\pm[n]$ with the property that $\sigma(\overline{i})=-\sigma(i)$ for all $i\in [n]$.
Let $\sigma=\sigma(1)\sigma(2)\cdots \sigma(n)\in\mbn$.
An {\it excedance} (resp.~{\it fixed point}) of $\sigma$ is an index $i\in [n]$ such that $\sigma(|\sigma(i)|)>\sigma(i)$ (resp.~$\sigma(i)=i$).
Let $\exc(\sigma)$ (resp.~$\fix(\sigma)$) denote the number of excedances (resp.~fixed points) of $\sigma$.
Let $\mdn^B=\{\sigma\in \mbn: \fix(\sigma)=0\}$ be the set of all derangements in $\mbn$.
The {\it type $B$ derangement polynomials} $d_n^B(x)$ are defined by
$$d_n^B(x)=\sum_{\sigma\in \mdn^B}x^{\exc(\sigma)},$$
which has been extensively studied, see~\cite{Chen09,Chow09} and references therein.
According to~\cite[Theorem 3.2]{Chow09}, the generating function of $d_n^B(x)$ is given as follows:
\begin{equation}\label{dxzB-EGF}
\sum_{n=0}^\infty d_n^B(x)\frac{z^n}{n!}=\frac{(1-x)\mathrm{e}^z}{\mathrm{e}^{2xz}-x\mathrm{e}^{2z}}.
\end{equation}
Chen, Tang and Zhao~\cite[Theorem 4.6]{Chen09}) studied the polynomials $x^nd_n^B(1/x)$ and proved the following remarkable result.
\begin{proposition}[\cite{Chen09}]\label{Chen09}
For $n\geqslant 1$, the polynomials $x^nd_n^B(1/x)$ are spiral. Equivalently, the polynomials $d_n^B(x)$ are alternatingly increasing.
\end{proposition}
As discussed in Sections~\ref{section004} and~\ref{section005}, many different refinements and generalizations of Propositions~\ref{Foata70},~\ref{Zeng12} and~\ref{Chen09} have been studied. The reader is referred to~\cite{Gessel20,Zeng2020,Lin15,Ma19,Zhuang17} for more details. In particular, Gessel and Zhuang~\cite{Gessel20} studied several permutation statistics jointly with the number of fixed points and jointly with cycle type.

Let us define the {\it $(p,q)$-Eulerian polynomials} $A_n(x,p,q)$ by
$$A_n(x,p,q)=\sum_{\pi\in\msn}x^{\exc(\pi)}p^{\fix(\pi)}q^{\cyc(\pi)}.$$
Let $A_n(x,q)=A_n(x,1,q)$ be the $q$-Eulerian polynomials.
Brenti~\cite{Brenti00} showed that some of the crucial properties of Eulerian polynomials have nice $q$-analogues for the polynomials $A_{n}(x,q)$.
Following~\cite[Proposition~7.2]{Brenti00}, the $q$-Eulerian polynomials $A_n(x,q)$ satisfy the recurrence
\begin{equation}\label{anxq-rr}
A_{n+1}(x,q)=(nx+q)A_{n}(x,q)+x(1-x)\frac{\mathrm{d}}{\mathrm{d} x}A_{n}(x,q),
\end{equation}
with the initial conditions $A_{0}(x;q)=1,A_{1}(x;q)=q$ and $A_{2}(x;q)=q(x+q)$.
%They satisfy the recurrence relation
%\begin{equation}\label{anxq-rr}
%A_{n+1}(x,q)=(nx+q)A_{n}(x,q)+x(1-x)\frac{d}{d x}A_{n}(x,q),
%\end{equation}
%with the initial conditions $A_{0}(x,q)=1,A_{1}(x,q)=q$ and $A_{2}(x,q)=q(x+q)$ (see~\cite[Proposition~7.2]{Brenti00}).
According to~\cite[Proposition~7.3]{Brenti00}, we have
\begin{equation*}\label{Anxq-egf01}
\sum_{n=0}^\infty A_n(x,q)\frac{z^n}{n!}=\left(\frac{(1-x)\mathrm{e}^z}{\mathrm{e}^{{xz}}-x\mathrm{e}^z}\right)^q.
\end{equation*}
Using the {\it exponential formula}, Ksavrelof and Zeng~\cite{Zeng02} found that
\begin{equation}\label{Anxpq-EGF}
\sum_{n=0}^\infty A_n(x,p,q)\frac{z^n}{n!}=\left(\frac{(1-x)\mathrm{e}^{pz}}{\mathrm{e}^{xz}-x\mathrm{e}^{z}}\right)^q
\end{equation}
Below are the polynomials $A_n(x,p,q)$ for $n\leqslant 4$:
\begin{align*}
A_1(x,p,q)&=pq,~
A_2(x,p,q)=p^2q^2+qx,~
A_3(x,p,q)=p^3q^3+(q+3pq^2)x+qx^2,\\
A_4(x,p,q)&=p^4q^4+(q+4pq^2+6p^2q^3)x+(4q+3q^2+4pq^2)x^2+qx^3.
\end{align*}
%The reader is referred to~\cite{Ma2008,Ma16,Zeng2020,Zeng16} for some recent results related to the joint distribution of excedances, fixed points and cycles.
%
From Proposition~\ref{Zeng12}, we see that if $q>0$ is a given real number, then $A_n(x,0,q)$ are $\gamma$-positive.
Comparing~\eqref{dxzB-EGF} with~\eqref{Anxpq-EGF}, we get
$2^nA_n(x,1/2,1)=d_n^B(x)$. It follows from Proposition~\ref{Chen09} that $A_n(x,1/2,1)$ are alternatingly increasing.
As a unified generalization of Propositions~\ref{Foata70},~\ref{Zeng12} and~\ref{Chen09},
we can now present the second main result of this paper.
\begin{theorem}\label{thm001}
Let $p\in[0,1]$ and $q\in[0,1]$ be two given real numbers, i.e., $0\leqslant p\leqslant 1$ and $0\leqslant q\leqslant 1$.
Then we have the following results:

{\it $(i)$} For $n\geqslant 1$, the polynomials $A_n(x,q)$ are bi-$\gamma$-positive;
%
%{\it $(ii)$} Let $(a_n(x,q),b_n(x,q))$ be the symmetric decomposition of $A_n(x,q)$. Then $a_n(x,q)$ and $b_n(x,q)$ have only real and non-positive zeros which are
%interlacing;

{\it $(ii)$} The polynomials $A_n(x,p,q)$ are alternatingly increasing.
\end{theorem}
The proofs of Theorems~\ref{bigammpartial} and~\ref{thm001} will be given in the next section. By combining a modified Foata-Strehl group action, we give a combinatorial interpretation of the partial $\gamma$-coefficients of $A_n(x,p,q)$.
In Section~\ref{bi-gamma-coefficients}, we give a combinatorial interpretation for the symmetric decomposition of $k^nA_n(x,1/k)$ as well as the bi-$\gamma$-coefficients of $2^nA_n(x,1/2)$, where $k$ is a fixed positive integer.
In Sections~\ref{section004} and~\ref{section005}, we study excedance and flag excedance statistics of signed and colored permutations.
In particular, we give a unified approach to the relationships between the $A_n(x,p,q)$ and some
multivariate Eulerian
polynomials, signed or colored, which can be summarized as follows:
\begin{align*}
B_n(x,y,s,t,p,q)&=(1+p)^ny^nA_n\left(\frac{x}{y},\frac{t+sp}{(1+p)y},q\right),\\
B_n^{(A)}(x,y,s,t,p,q)&=(1+p)^ny^nA_n\left(\frac{x+py}{(1+p)y},\frac{t+sp}{(1+p)y},q\right),\\
A_{n,r}(x,y,s,t,p,q)&=[r]_p^ny^nA_n\left(\frac{x}{y},\frac{t+sp[r-1]_p}{[r]_py},q\right),\\
A_{n}^{(r)}(x,y,s,t,p,q)&=[r]_p^ny^nA_n\left(\frac{x+p[r-1]_py}{[r]_py},\frac{t+sp[r-1]_p}{[r]_py},q\right),
\end{align*}
where $[r]_p=1+p+\cdots+p^{r-1}$ for a fixed positive integer $r$ and $[0]_p=0$.
%A variety of known formulas are recovered as special cases of our results.

In the following discussion, we always write permutation, signed or not, colored or not, by using its standard cycle form,
in which each cycle has its smallest (in absolute value) element first and the cycles are written in increasing order of the absolute value of their first elements.
%%%%%%%%%%%%%%%%%%%%%%%%%%%%%%%%%%%%%%%%%%%%%%%%%%%%%%%%%%%%%%%%%%%%%%%%%%%%%%%%%
%%%%%%%%%%%%%%%%%%%%%%%%%%%%%%%%%%%%%%%%%%%%%%%%%
\section{Proof of Theorems~\ref{bigammpartial} and~\ref{thm001}}\label{section002}
%%%%%%%%%%%%%%%%%%%%%%%%%%%%%%%%%%%%%%%%%%%
%%%%%%%%%%%%%%%%%%%%%%%%%%%%%%%%%%%%%%%%%%%%%%%%%%%%%%%%%%%%%%%%%%%%%%%%%%%%%%%%%
%%%%%%%%%%%%%%%%%%%%%%%%%%%%%%%%%%%%%%%%%%%%%%%%%%%%%%%%%%%%%%%%%%%%%%%%%%%%%%%%%%
%%%%%%%%%%%%%%%%%%%%%%%%%%%%%%%%%%%%%%%%%%%%%%%%%%%%%%%%%%%%%%%%%%%%%%%%%%%%%%%%%%
%%%%%%%%%%%%%%%%%%%%%%%%%%%%%%%%%%%%%%%%%%%%%%%%%%
%%%%%%%%%%%%%%%%%%%%%%%%%%%%%%%%%%%%%%%%%%%%%%%%%%%%%%%%%%%%%%%%%%%%%%%%%%%%%%%%%%
%\begin{equation*}
%p(x,y)=\sum_{i=0}^ny^i\sum_{j=0}^{\lrf{(n-i)/2}}\mu_{n,i,j}x^j(1+x)^{n-i-2j}.
%\end{equation*}
%
\noindent{\bf A proof
Theorem~\ref{bigammpartial}:}
\begin{proof}
For $0\leqslant i\leqslant n$ and $0\leqslant j\leqslant \lrf{(n-i)/2}$, let
$$\mu_{n,i,j}x^j(1+x)^{n-i-2j}=\sum_{\ell=j}^{n-i-j}S_{n,i,j,\ell}x^{\ell}.$$
Since the polynomials $\sum_{\ell=j}^{n-i-j}S_{n,i,j,\ell}x^{\ell}$ are symmetric and unimodal, we have
\begin{equation}\label{eq:proof40}
\left\{
  \begin{array}{ll}
    S_{n,i,j,\ell}=S_{n,i,j,n-i-\ell}, & \hbox{if $j\leqslant \ell\leqslant n-i-j$;} \\
     S_{n,i,j,\ell}\leqslant  S_{n,i,j,\ell+k}, & \hbox{if $j\leqslant \ell<\ell+k\leqslant \lrf{(n-i)/2}$;} \\
      S_{n,i,j,\ell}\geqslant  S_{n,i,j,\ell+k}, & \hbox{if $\lrf{(n-i)/2}\leqslant \ell<\ell+k\leqslant n-i-j$.}
  \end{array}
\right.\end{equation}
For $n\geqslant 1$, assume that $p(x,y)=\sum_{\ell=0}^{n-1}p_{n,\ell}(y)x^{\ell}$
where
$$p_{n,\ell}(y)=\sum_{i=0}^ny^i\sum_{j=0}^{\lrf{(n-i)/2}}S_{n,i,j,\ell}.$$

For $0\leqslant \ell\leqslant \lrf{n/2}$, we have
\begin{align*}
p_{n,\ell}(y)-p_{n,n-\ell}(y)&=\sum_{i=0}^ny^i\sum_{j=0}^{\lrf{(n-i)/2}}\left(S_{n,i,j,\ell}-S_{n,i,j,n-\ell}\right)\\
&=\sum_{i=0}^ny^i\sum_{j=0}^{\lrf{(n-i)/2}}\left(S_{n,i,j,\ell}-S_{n,i,j,\ell-i}\right).
\end{align*}
From~\eqref{eq:proof40}, we have $S_{n,i,j,\ell}-S_{n,i,j,\ell-i}\geqslant 0$. Hence $p_{n,\ell}(y)\geqslant p_{n,n-\ell}(y)$ when $y\geqslant 0$.

For $0\leqslant \ell\leqslant \lrf{n/2}-1$, we have
\begin{align*}
p_{n,n-1-\ell}(y)-p_{n,\ell}(y)&=\sum_{i=0}^ny^i\sum_{j=0}^{\lrf{(n-i)/2}}\left(S_{n,i,j,n-1-\ell}-S_{n,i,j,\ell}\right)\\
&=\sum_{i=0}^ny^i\sum_{j=0}^{\lrf{(n-i)/2}}\left(S_{n,i,j,\ell+1-i}-S_{n,i,j,\ell}\right),
\end{align*}
which can be rewritten as
$p_{n,n-1-\ell}(y)-p_{n,\ell}(y)=P_{n,\ell}-Q_{n,\ell}(y)$,
where
\begin{align*}
P_{n,\ell}&=\sum_{j=0}^{\lrf{n/2}}\left(S_{n,0,j,\ell+1}-S_{n,0,j,\ell}\right),~
Q_{n,\ell}(y)=\sum_{i=1}^ny^i\sum_{j=0}^{\lrf{(n-i)/2}}\left(S_{n,i,j,\ell}-S_{n,i,j,\ell+1-i}\right).
\end{align*}
It follows from~\eqref{eq:proof40} that $S_{n,0,j,\ell+1}-S_{n,0,j,\ell}\geqslant 0$ and $S_{n,i,j,\ell}-S_{n,i,j,\ell+1-i}\geqslant 0$ for $i\geqslant 1$.
Since $p(x,1)=\sum_{\ell=0}^{n-1}p_{n,\ell}(1)x^{\ell}$ is bi-$\gamma$-positive, the polynomial $p(x,1)$ is alternatingly increasing, which implies that
\begin{equation}\label{an1q}
p_{n,n-1-\ell}(1)-p_{n,\ell}(1)=P_{n,\ell}-Q_{n,\ell}(1)\geqslant 0.
\end{equation}
Therefore, if $0\leqslant y\leqslant 1$, then $P_{n,\ell}-Q_{n,\ell}(y)\geqslant P_{n,\ell}-Q_{n,\ell}(1)\geqslant 0$.
%$$a_{n,n-1-\ell}(p,q)-a_{n,\ell}(p,q)=P_{n,\ell}(q)-Q_{n,\ell}(p,q)\geqslant P_{n,\ell}(q)-Q_{n,\ell}(1,q)\geqslant 0.$$
In conclusion, when $0\leqslant y\leqslant 1$, we get $p_{n,\ell}(y)\geqslant p_{n,n-\ell}(y)$ for $0\leqslant \ell\leqslant \lrf{n/2}$ and $p_{n,n-1-\ell}(y)\geqslant p_{n,\ell}(y)$ for $0\leqslant \ell\leqslant \lrf{n/2}-1$. This completes the proof.
\end{proof}

In the rest of this section, we shall prove Theorem~\ref{thm001} by using the theory of context-free grammars.
For an alphabet $V$, let $\mathbb{Q}[[V]]$ be the rational commutative ring of formal power series in
monomials formed from letters in $V$. A {\it context-free grammar} over
$V$ is a function $G: V\rightarrow \mathbb{Q}[[V]]$ that replaces a letter in $V$ by an element of $\mathbb{Q}[[V]]$ (see~\cite{Chen93,Ma19}).
The formal derivative $D_G$ is a linear operator defined with respect to a grammar $G$. Following~\cite{Chen17},
a {\it grammatical labeling} is an assignment of the underlying elements of a combinatorial structure
with variables, which is consistent with the substitution rules of a grammar.
%In other words, $D_G$ is the unique derivation satisfying $D_G(x)=G(x)$ for $x\in V$.
For any two formal functions $u$ and $v$, we have
$D_G(u+v)=D_G(u)+D_G(v),~D_G(uv)=D_G(u)v+uD_G(v)$.
For a constant $c$, we have $D_G(c)=0$.
%The {\it Leibniz rule} is as follows:
%\begin{equation*}\label{Dnuv}
%D_G^n(uv)=\sum_{k=0}^n\binom{n}{k}D_G^k(u)D_G^{n-k}(v).
%\end{equation*}
\begin{example}
If $G=\{x\rightarrow xy, y\rightarrow xy\}$, then $D_G(x)=xy,D_G(y)=xy,D_G^2(x)=xy(x+y)$.
\end{example}

Recall that $\exc(\pi)+\drop(\pi)+\fix(\pi)=n$ for $\pi\in\msn$.
The following lemma is fundamental.
\begin{lemma}\label{lemma001exc}
Let $G=\{I\rightarrow Ipq, p\rightarrow xy, x\rightarrow xy, y\rightarrow xy\}$.
Then
$$D_G^n(I)=I\sum_{\pi\in\msn}x^{\exc(\pi)}y^{\drop(\pi)}p^{\fix(\pi)}q^{\cyc(\pi)}.$$
\end{lemma}
\begin{proof}
Let $\pi\in \msn$. We introduce a grammatical labeling of $\pi$:
\begin{itemize}
  \item [\rm ($L_1$)]If $i$ is an excedance, then put a superscript label $x$ right after $i$;
  \item [\rm ($L_2$)]If $i$ is a drop, then put a superscript label $y$ right after $i$;
 \item [\rm ($L_3$)]If $i$ is a fixed point, then put a superscript label $p$ right after $i$;
 \item [\rm ($L_4$)]Put a superscript label $I$ right after $\pi$ and put a subscript label $q$ right after each cycle.
\end{itemize}
With this labeling, the weight of $\pi$ is
defined as the product of its labels, that is,
$$w(\pi)=Ix^{\exc(\pi)}y^{\drop(\pi)}p^{\fix(\pi)}q^{\cyc(\pi)}.$$
For example, let $\pi=(1,4,3)(2,6)(5)$. The grammatical labeling of $\pi$ is given below
$$(1^x4^y3^y)_q(2^x6^y)_q(5^p)^I_q.$$
If we insert $7$ after $5$, the resulting permutation is
$(1^x4^y3^y)_q(2^x6^y)_q(5^x7^y)^I_q$.
If the inserted $7$ forms a new cycle, the resulting permutation is
$(1^x4^y3^y)_q(2^x6^y)_q(5^p)_q(7^p)^I_q$.
If we insert $7$ after $1$ or $4$, then we respectively get
$(1^x7^y4^y3^y)_q(2^x6^y)_q(5^p)^I_q$ and $(1^x4^x7^y3^y)_q(2^x6^y)_q(5^p)^I_q$.
In each case, the insertion of $7$ corresponds to one substitution rule in $G$.
By induction, it is routine to verify that the action of $D_G$
on the set of weights of permutations in $\msn$ gives the set of weights of permutations in $\ms_{n+1}$. This completes the proof.
\end{proof}

In the rest part of this section, we always let $k$ be a fixed positive integer.
Following~\cite{Savage12}, the {\it $1/k$-Eulerian polynomials} $A_{n}^{(k)}(x)$ are defined as follows:
\begin{equation}\label{Ankx-def01}
\sum_{n=0}^\infty A_{n}^{(k)}(x)\frac{z^n}{n!}=\left(\frac{1-x}{\mathrm{e}^{kz(x-1)}-x} \right)^{\frac{1}{k}},
\end{equation}
Savage and Viswanathan~\cite[Section~1.5]{Savage12} proved that
$A_{n}^{(k)}(x)$ is the {\it $\rm s$-Eulerian polynomial} of the {\it$\rm s$-inversion sequence} $(1,k+1,2k+1,\ldots,(n-1)k+1)$.
By comparing~\eqref{Anxpq-EGF} with~\eqref{Ankx-def01}, one get
\begin{equation}\label{Ankx-def02}
A_{n}^{(k)}(x)=k^nA_n(x,1,1/k)=k^nA_n(x,1/k)=\sum_{\pi\in\msn}x^{\exc(\pi)}k^{n-\cyc(\pi)}.
\end{equation}
Very recently, a bijective proof of~\eqref{Ankx-def02} was provided in~\cite{Chao19}.

For $n\geqslant 1$, we let $A_n^{(k)}(x)=\sum_{j=0}^{n-1}A_{n,j;k}x^j$.
The first few $A_n^{(k)}(x)$ are
$A_1^{(k)}(x)=1,~A_2^{(k)}(x)=1+kx,~A_3^{(k)}(x)=1+3kx+k^2x(1+x)$.
Comparing~\eqref{Ankx-def02} with~\eqref{anxq-rr}, we immediately get that
\begin{equation}\label{Anj-recu}
A_{n+1,j;k}=(1+kj)A_{n,j;k}+k(n-j+1)A_{n,j-1;k},
\end{equation}
with the initial conditions $A_{1,0;k}=1$ and $A_{1,i;k}=0$ for $i\neq0$.
Now we give a grammatical description of the numbers $A_{n,j;k}$.
\begin{lemma}\label{marked-per}
If
$G_0=\{I\rightarrow Iy,~x\rightarrow kxy,~y\rightarrow kxy\}$,
then for $n\geqslant 1$, we have
\begin{equation}\label{DnI-Stirl01}
D_{G_0}^n(I)=I\sum_{j=0}^{n-1}A_{n,j;k}x^jy^{n-j}.
\end{equation}
\end{lemma}
\begin{proof}
Note that $D_{G_0}(I)=Iy,D_{G_0}^2(I)=I(y^2+kxy)$.
Hence the result holds for $n=1,2$.
Now assume that~\eqref{DnI-Stirl01} holds for some $n$, where $n\geqslant 2$.
Note that
\begin{align*}
D_{G_0}^{n+1}(I)&=\sum_{j}A_{n,j;k}I\left(x^jy^{n-j+1}+kjx^jy^{n-j+1}+k(n-j)x^{j+1}y^{n-j}\right).
\end{align*}
Taking coefficients of $x^jy^{n-j+1}$ on both sides of the above expression yields
the recurrence relation~\eqref{Anj-recu}. The proof follows by induction.
\end{proof}

Following~\cite{Ma19}, a {\it change of grammar} is a substitution method in which the original grammar is replaced with functions of other grammars, which has proved to be useful in handling
combinatorial expansions. In~\cite{Ma19}, the change of grammar method was defined
and used to prove the $\gamma$-positivity of some descent-type polynomials. In this paper, we shall use the change of grammar technique to prove some of the main results.

\noindent{\bf A proof
Theorem~\ref{thm001}:}
\begin{proof}
The proof is divided into two parts.

{{$\bf {(i)}$}} Let $q\in [0,1]$ be a given real number and let $k$ be a given positive integer.
From~\eqref{Ankx-def02}, we see that $A_n^{(k)}(x)=k^nA_n(x,1/k)$. To prove the bi-$\gamma$-positivity of $A_n(x,q)$, it suffices to prove that the polynomials $A_n^{(k)}(x)$ are bi-$\gamma$-positive.
Consider a change of the grammar given in Lemma~\ref{marked-per}.
Note that
$$D_{G_0}(I)=Iy,~
D_{G_0}(Iy)=Iy(x+y)+(k-1)Ixy,$$
$$D_{G_0}(x+y)=2kxy,~
D_{G_0}(xy)=kxy(x+y).$$
Set $J=Iy,u=x+y$ and $v=xy$. Then
$$D_{G_0}(I)=J,~D_{G_0}(J)=Ju+(k-1)Iv,~D_{G_0}(u)=2kv,~D_{G_0}(v)=kuv.$$
Let $G_1=\{I\rightarrow J,J\rightarrow Ju+(k-1)Iv, u\rightarrow 2kv, v\rightarrow kuv\}$.
By induction, it is routine to verify that there are nonnegative integers such that
\begin{align}\label{DnI-stir}
D_{G_1}^n(I)&=J\sum_{i=0}^{\lrf{(n-1)/2}}A^+_{n,i;k}v^iu^{n-1-2i}+Iv\sum_{i=0}^{\lrf{(n-2)/2}} A^-_{n,i;k}v^iu^{n-2-2i}.
\end{align}
In particular, $D_{G_1}(I)=J,D_{G_1}^2(I)=Ju+(k-1)Iv$.
%$$D(I)=J,D^2(I)=Ju+(k-1)Iv,D^3(I)=J(u^2+(3k-1)v)+I(k^2-1)uv.$$
%Hence $N^+_{1,0;k}=N^+_{2,0;k}=1,N^-_{1,0;k}=0$ and $N^-_{2,0;k}=k-1$.
Note that
\begin{align*}
D_{G_1}^{n+1}(I)
%&=D\left(J\sum_{i} N^+_{n,i;k}v^iu^{n-1-2i}+Iv\sum_{i} N^-_{n,i;k}v^iu^{n-2-2i}\right)\\
&=(Ju+(k-1)Iv)\sum_{i} A^+_{n,i;k}v^iu^{n-1-2i}+Jv\sum_{i} A^-_{n,i;k}v^iu^{n-2-2i}\\
&J\sum_{i} A^+_{n,i;k} \left(kiv^iu^{n-2i}+2k(n-1-2i)v^{i+1}u^{n-2-2i}\right)+\\
&I\sum_{i} A^-_{n,i;k}\left(k(i+1)v^{i+1}u^{n-1-2i}+2k(n-2-2i)v^{i+2}u^{n-3-2i}\right).
\end{align*}
Taking coefficients of $Jv^iu^{n-2i}$ and $Iv^{i+1}u^{n-1-2i}$ on both sides yields the recurrence system
\begin{equation}\label{recusystem}
\left\{
  \begin{array}{ll}
    A^+_{n+1,i;k}=(1+ki)A^+_{n,i;k}+2k(n-2i+1)A^+_{n,i-1;k}+A^-_{n,i-1;k}, & \\
   A^-_{n+1,i;k}=k(i+1)A^-_{n,i;k}+2k(n-2i)A^-_{n,i-1;k}+(k-1)A^+_{n,i;k}, &
  \end{array}
\right.
\end{equation}
with $A_{1,0;k}^+=1$, $A_{1,i;k}^+=0$ for $i\neq0$ and $A_{1,i;k}^-=0$ for any $i$.
Clearly, $A^+_{n,i;k}$ and $A^-_{n,i;k}$ are both nonnegative when $k\geqslant1$.
Taking $J=Iy,u=x+y$ and $v=xy$ in~\eqref{DnI-stir}, and then comparing~\eqref{DnI-Stirl01} with~\eqref{DnI-stir}, we obtain
\begin{equation}\label{Ankx-decom}
A_n^{(k)}(x)=a_n^{(k)}(x)+xb_n^{(k)}(x),
\end{equation}
where
\begin{equation}\label{ankxbnkx}
\left\{
  \begin{array}{ll}
    a_n^{(k)}(x)=\sum_{i\geqslant0}A^+_{n,i;k}x^i(1+x)^{n-1-2i}, & \\
    b_n^{(k)}(x)=\sum_{i\geqslant0}A^-_{n,i;k}x^i(1+x)^{n-2-2i}. &
  \end{array}
\right.
\end{equation}
%
%\begin{equation}\label{ankxbnkx}
%a_n^{(k)}(x)=\sum_{i=0}^{\lrf{{(n-1)}/{2}}}A^+_{n,i;k}x^i(1+x)^{n-1-2i},~b_n^{(k)}(x)=\sum_{i=0}^{\lrf{{(n-2)}/{2}}}A^-_{n,i;k}x^i(1+x)^{n-2-2i}.
%\end{equation}
%$N_{n;k}(x)=N_{n;k}^+(x)+xN_{n;k}^-(x)$, where
%\begin{align*}
%N^+_{n;k}(x)&=\sum_{i=0}^{\lrf{{(n-1)}/{2}}}N^+_{n,i;k}x^i(1+x)^{n-1-2i},
%N_n^-(x)=\sum_{i=0}^{\lrf{{n-2}/{2}}}N^-_{n,i;k}x^i(1+x)^{n-2-2i}.
%\end{align*}
By using~\eqref{recusystem}, it is routine to derive the following recurrence system:
\begin{equation*}\label{Stirling-recu4}
\left\{
  \begin{array}{ll}
a_{n+1}^{(k)}(x)&=(1+x+k(n-1)x)a_{n}^{(k)}(x)+kx(1-x)\frac{\mathrm{d}}{\mathrm{d}x}a_{n}^{(k)}(x)+xb_{n}^{(k)}(x),\\
b_{n+1}^{(k)}(x)&=k(1+(n-1)x)b_{n}^{(k)}(x)+kx(1-x)\frac{\mathrm{d}}{\mathrm{d}x}b_{n}^{(k)}(x)+(k-1)a_{n}^{(k)}(x),
  \end{array}
\right.
\end{equation*}
with $a_{1}^{(k)}(x)=1$ and $b_{1}^{(k)}(x)=0$.
%Let $(a_n(x,q),b_n(x,q))$ be the symmetric decomposition of $A_n(x,q)$.
%It follows from~\eqref{Ankx-decom} that
%\begin{equation}\label{ankxbnkx01}
%a_n^{(k)}(x)=k^na_n(x,1/k),~b_n^{(k)}(x)=k^nb_n(x,1/k).
%\end{equation}
In conclusion, the polynomials $A_n^{(k)}(x)$ are bi-$\gamma$-positive, and so $A_n(x,q)$ are bi-$\gamma$-positive if $q\in[0,1]$.
This completes the proof of the first statement.

{ $\bf{(ii)}$} %We first derive an expansion of $A_n(x,p,q)$.
We first consider a change of the grammar given in Lemma~\ref{lemma001exc}. Setting $u=xy$ and $v=x+y$, we get
$$D_G(I)=Ipq,~D_G(p)=u,~D_G(u)=uv,~D_G(v)=2u.$$ Let
$G_2=\{I\rightarrow Ipq, p\rightarrow u, u\rightarrow uv,  v\rightarrow 2u\}$.
By induction, it is routine to verify that
\begin{equation}\label{eq:proof1}
D_{G_2}^n(I)=I\sum_{i=0}^np^i\sum_{j=0}^{\lrf{(n-i)/2}}\gamma_{n,i,j}(q)u^jv^{n-i-2j}.
\end{equation}
In particular, $D_{G_2}(I)=Ipq,D_{G_2}^2(I)=I(p^2q^2+qu),D_{G_2}^3(I)=I(p^3q^3+3pq^2u+quv)$.
Therefore,
\begin{align*}
&D_{G_2}^{n+1}(I)=D_{G_2}\left(I\sum_{i,j}\gamma_{n,i,j}(q)p^iu^jv^{n-i-2j}\right)\\
&=I\sum_{i,j}\gamma_{n,i,j}(q)\left(qp^{i+1}u^{j}v^{n-i-2j}+ip^{i-1}u^{j+1}v^{n-i-2j}+jp^iu^jv^{n+1-i-2j}\right)+\\
&I\sum_{i,j}\gamma_{n,i,j}(q)(2n-2i-4j)p^{i}u^{j+1}v^{n-1-i-2j},
\end{align*}
which yields that the numbers $\gamma_{n,i,j}(q)$ satisfy the recurrence relation
\begin{equation}\label{Enij-recu}
\gamma_{n+1,i,j}(q)=q\gamma_{n,i-1,j}(q)+(i+1)\gamma_{n,i+1,j-1}(q)+j\gamma_{n,i,j}(q)+(2n-2i-4j+4)\gamma_{n,i,j-1}(q),
\end{equation}
with the initial conditions $\gamma_{1,1,0}(q)=q$ and $\gamma_{1,i,j}(q)=0$ for $(i,j)\neq(1,0)$. From~\eqref{Enij-recu}, we see that if $q\geqslant0$, then $\gamma_{n,i,j}(q)\geqslant0$.
Moreover, upon taking $u=xy$ and $v=x+y$ in~\eqref{eq:proof1}, we get the following expansion:
\begin{equation*}\label{eq:proof2}
D_{G}^n(I)=I\sum_{i=0}^np^i\sum_{j=0}^{\lrf{(n-i)/2}}\gamma_{n,i,j}(q)(xy)^j(x+y)^{n-i-2j}.
\end{equation*}
Setting $I=y=1$, we obtain $D_{G}^n(I)\mid_{I=y=1}=A_n(x,p,q)$ and so
\begin{equation}\label{eq:proof03}
A_n(x,p,q)=\sum_{i=0}^np^i\sum_{j=0}^{\lrf{(n-i)/2}}\gamma_{n,i,j}(q)x^j(1+x)^{n-i-2j}.
\end{equation}
Therefore, if $q\geqslant 0$ is a given real number, then $A_n(x,p,q)$ is partial $\gamma$-positive.
From the first part of the proof, we see that when $0\leqslant q\leqslant 1$, the polynomial
$A_n(x,1,q)$ is bi-$\gamma$-positive. In conclusion, if $0\leqslant p\leqslant 1$ and $0\leqslant q\leqslant 1$ are both given real numbers, it follows from Theorem~\ref{bigammpartial} that
the polynomial $A_n(x,p,q)$ is alternatingly increasing.
This completes the proof.
\end{proof}

As a generalization of Proposition~\ref{Zeng12}, we now give the following result.
\begin{proposition}\label{prop14}
Let $\gamma_{n,i,j}(q)$ be the polynomials defined by the recurrence relation~\eqref{Enij-recu}.
Then
\begin{equation}\label{bnij-combin}
\gamma_{n,i,j}(q)=\sum_{\pi\in \ms_{n,i,j}}q^{\cyc(\pi)}.
\end{equation}
where $\ms_{n,i,j}=\{\pi\in\msn: \cda(\pi)=0,~\fix(\pi)=i,~\exc(\pi)=j\}$.
Therefore, the number $\gamma_{n,i,j}(1)$ counts permutations in $\msn$ with no cycle double ascents, $i$ fixed pints and $j$ excedances.
\end{proposition}

Let $(c_1,c_2,\ldots,c_i)$ be a cycle of $\pi$. Then $c_1=\min\{c_1,\ldots,c_i\}$. Set $c_{i+1}=c_1$. Then $c_j$ is called
\begin{itemize}
  \item a {\it cycle double ascent} in the cycle if $c_{j-1}<c_j<c_{j+1}$, where $2\leqslant j\leqslant i-1$;
  \item a {\it cycle double descent} in the cycle if $c_{j-1}>c_j>c_{j+1}$, where $2<j\leqslant i$;
 \item a {\it cycle peak} in the cycle if $c_{j-1}<c_j>c_{j+1}$, where $2\leqslant j \leqslant i$;
 \item a {\it cycle valley} in the cycle if $c_{j-1}>c_j<c_{j+1}$, where $2<j\leqslant i-1$.
\end{itemize}
We define an action $\varphi_{x}$ on $\msn$ as follows.
Let $c=(c_1,c_2,\ldots,c_i)$ be a cycle of $\pi\in \msn$ with at least two elements.
Consider the following three cases:
\begin{itemize}\item  If $c_k$ is a cycle double ascent in $c$,
then $\varphi_{c_k}(\pi)$ is obtained by deleting $c_k$ and then inserting $c_k$ between $c_j$ and $c_{j+1}$,
where $j$ is the smallest index satisfying $k< j\leqslant i$ and $c_j>c_k>c_{j+1}$;
\item If $c_k$ is a cycle double descent in $c$, then $\varphi_{c_k}(\pi)$ is obtained by deleting $c_k$ and then inserting $c_k$ between $c_j$ and $c_{j+1}$, where $j$ is the largest index satisfying $1\leqslant j<k$ and $c_j<c_k<c_{j+1}$;
\item If $c_k$ is neither a cycle double ascent nor a cycle double descent in $c$, then $c_k$ is a cycle peak or a cycle valley.
In this case, we let $\varphi_{c_k}(\pi)=\pi$.
\end{itemize}

Following~\cite{Branden08}, we now define a {\it modified Foata-Strehl group action} $\varphi'_x$ on $\msn$ by
$$\varphi'_x(\pi)=\left\{\begin{array}{lll}
\varphi_x(\pi),&\text{ if $x$ is a cycle double ascent or a cycle double descent;}\\
\pi,&\text{if $x$ is a cycle peak or a cycle valley.}\\
\end{array}\right.$$

Define
$$\CDD(\pi)=\{x \mid x\text{~is a cycle double descent of $\pi$}\},$$
$$\ms_{n,i,j,k}^1=\{\pi\in\msn: \cda(\pi)=0,~\fix(\pi)=i,~\exc(\pi)=j,~\cyc(\pi)=k\},$$
$${\ms}_{n,i,j,k}^2=\{\pi\in\msn: \cda(\pi)=1,~\fix(\pi)=i,~\exc(\pi)=j,~\cyc(\pi)=k\}.$$
For $\pi\in\ms_{n,i,j,k}^1$ and $x\in \CDD(\pi)$, it should be noted that $\exc(\pi)$ equals the number of cycle peaks of $\pi$, $\varphi_{x}'(\pi)\in{\ms}_{n,i,j+1,k}^2$ and $x$ is the unique cycle double ascent of $\varphi_{x}'(\pi)$.
Conversely, for $\pi\in{\ms}_{n,i,j+1,k}^2$, let $x$ be the unique cycle double ascent of $\pi$.
Note that $\varphi_{x}'(\pi)\in \ms_{n,i,j,k}^1$ and $x$ becomes a cycle double descent in $\varphi'_{x}(\pi)$. This implies that $$|{\ms}_{n,i,j+1,k}^2|=(n-i-2j)|\ms_{n,i,j,k}^1|,$$
where $n-i-2j$ is the number of cycle double descents of permutations in $\ms_{n,i,j,k}^1$.

\begin{example}
Let $\pi=(1,10,6,5,7,3,2,8)(4,9)\in\ms_{10,0,4,2}^1$. We have $\CDD(\pi)=\{3,6\}$. Then
\begin{align*}
\varphi_{3}'(\pi)&=(1,3,10,6,5,7,2,8)(4,9),~
\varphi_{6}'(\pi)=(1,6,10,5,7,3,2,8)(4,9),
\end{align*}
and $\varphi_{3}'(\pi),~\varphi_{6}'(\pi)\in\ms_{10,0,5,2}^2$.
\end{example}

\noindent{\bf A proof of Proposition~\ref{prop14}:}
\begin{proof}
In order to get a permutation enumerated by $\gamma_{n+1,i,j}(q)$, we distinguish five cases:
\begin{enumerate}
\item [($c_1$)] If $\pi\in \ms_{n,i-1,j}$, then we need append $(n+1)$ to $\pi$ as a new cycle. This accounts for the term $q\gamma_{n,i-1,j}(q)$;
\item [($c_2$)] If $\pi\in \ms_{n,i+1,j-1}$, then we should insert the entry $n+1$ right after a fixed point. This accounts for the term $(1+i)\gamma_{n,i+1,j-1}(q)$;
 \item [($c_3$)] If $\pi\in \ms_{n,i,j}$, then we should insert the entry $n+1$ right after an excedance.
This accounts for the term $j\gamma_{n,i,j}(q)$;
 \item [($c_4$)] If $\pi\in \ms_{n,i,j-1,k}^1$, then there are $n-i-2(j-1)$ positions could be inserted the entry $n+1$, since
we cannot insert the entry $n+1$ immediately before or right after each cycle peak, and we cannot insert the entry $n+1$ right after a fixed point.
This accounts for the term $(n-i-2j+2)\gamma_{n,i,j-1}(q)$;
 \item [($c_5$)]  If $\pi\in \ms_{n,i,j-1,k}^1$, let $x$ be one cycle double descent of $\pi$. Note that
$\varphi'_{x}(\pi)\in \ms_{n,i,j,k}^2$ and $x$ becomes the unique cycle double ascent. We should insert the entry $n+1$ into $\varphi'_{x}(\pi)$ immediately before $x$.
This accounts for the term $(n-i-2j+2)\gamma_{n,i,j-1}(q)$.
 \end{enumerate}
Thus~\eqref{Enij-recu} holds. This completes the proof.
\end{proof}

Let $\gamma_{n,i,j}(q)$ be the polynomials defined by the recurrence relation~\eqref{Enij-recu}.
Define
\begin{equation*}\label{eq:proof3}
\gamma=\gamma(x,p,q;z)=1+\sum_{n=1}^\infty\sum_{i=0}^np^i\sum_{j=0}^{\lrf{(n-i)/2}}\gamma_{n,i,j}(q)x^j\frac{z^n}{n!}.
\end{equation*}
\begin{proposition}
The generating function $\gamma(x,p,q;z)$ is given by
\begin{equation}\label{Csxz-explicit}
\gamma(x,p,q;z)=e^{z\left(p-\frac{1}{2}\right)q}\left(\frac{\sqrt{1-4x}}{\sqrt{1-4x}\cosh\left(\frac{z}{2}\sqrt{1-4x}\right)-\sinh\left(\frac{z}{2}\sqrt{1-4x}\right)}\right)^q.
\end{equation}
%\begin{equation}\label{Csxz-explicit}
%\gamma(x,p,q;z)=e^{z\left(p-\frac{1}{2}\right)q}\left(\frac{\sqrt{1-4x}}
%{\sqrt{1-4x}\cosh\left(\frac{z}{2}\sqrt{1-4x}\right)-\sinh\left(\frac{z}{2}\sqrt{1-4x}\right)\right)}\right)^q.
%\end{equation}
\end{proposition}
\begin{proof}
By rewriting~\eqref{Enij-recu} in terms of the generating function $\gamma=\gamma(x,p,q;z)$, we have
\begin{equation}\label{recu09}
\frac{\partial \gamma}{\partial z}=pq\gamma+x(1-2p)\frac{\partial \gamma}{\partial p}+x(1-4x)\frac{\partial \gamma}{\partial x}+2xz\frac{\partial \gamma}{\partial z}.
\end{equation}
It is routine to check that the generating function
$$
\widehat{\gamma}(x,p,q;z)=e^{z\left(p-\frac{1}{2}\right)q}\left(\frac{\sqrt{1-4x}}{\sqrt{1-4x}\cosh\left(\frac{z}{2}\sqrt{1-4x}\right)-\sinh\left(\frac{z}{2}\sqrt{1-4x}\right)}\right)^q$$
satisfies~\eqref{recu09}. Also, this generating function gives $\widehat{\gamma}(x,p,q;0)=\widehat{\gamma}(x,p,0;z)=1$.
Hence $\widehat{\gamma}(x,p,q;z)=\gamma(x,p,q;z)$. This completes the proof.
\end{proof}

The Springer numbers are introduced by Springer~\cite{Springer71} in the study of irreducible root
system of type $B_n$.
Let $s_n$ denote the $n$-th Springer number. Springer~\cite{Springer71} derived that
\begin{equation}\label{sn}
\sum_{n=0}^{\infty}s_n\frac{z^n}{n!}=\frac{1}{\cos(z)-\sin(z)}.
\end{equation}
Arnold~\cite{Arnold} found a combinatorial interpretation of the Springer numbers in terms of snakes.
The reader is referred to~\cite[A001586]{Sloane} for more combinatorial interpretations of $s_n$.
It follows from~\eqref{Csxz-explicit} that
\begin{equation}\label{122z}
\gamma(1/2,1,1;2z)=\frac{\mathrm{e}^z}{\cos(z)-\sin(z)}.
\end{equation}
Let $\mcc_n$ be the set of permutations in $\msn$ with no cycle double ascents.
Combining~\eqref{sn} and~\eqref{122z}, one can immediately get the following.
\begin{corollary}
For $n\geqslant 1$, we have $$\sum_{\pi\in\mcc_n}2^{n-\exc(\pi)}=\sum_{i=0}^n\binom{n}{i}s_i.$$
\end{corollary}

A {\it left peak} of $\pi\in\msn$ is an index $i\in[n-1]$ such that $\pi(i-1)<\pi(i)>\pi(i+1)$, where $\pi(0)=0$.
Denote by $\lpk(\pi)$ the number of left peaks in $\pi$.
Let $Q(n,i)=\#\{\pi\in\msn: \lpk(\pi)=i\}$, and let $Q_n(x)=\sum_{i=0}^{\lrf{n/2}}Q(n,i)x^i$.
Gessel~\cite[A008971]{Sloane} obtained the following exponential generating function:
\begin{equation}\label{Gessel}
Q(x;z)=1+\sum_{n=1}^\infty Q_n(x)\frac{z^n}{n!}=\frac{\sqrt{1-x}}{\sqrt{1-x}\cosh(z\sqrt{1-x})-\sinh(z\sqrt{1-x})}.
\end{equation}
Comparing~\eqref{Csxz-explicit} with~\eqref{Gessel} leads to
\begin{equation}\label{Qxz}
Q(x;z)=\gamma\left(\frac{x}{4},\frac{1}{2},1;2z\right).
\end{equation}
As an equivalent result of~\eqref{Qxz}, we now present
the following result.
\begin{corollary}
We have
\begin{equation*}
\sum_{\pi\in\msn}x^{\lpk(\pi)}=\sum_{\pi\in\mcc_n}2^{n-\fix(\pi)-2\exc(\pi)}x^{\exc(\pi)}.
\end{equation*}
\end{corollary}

Let $$\gamma_{n}(x,p,q)=\sum_{i=0}^n\sum_{j=0}^{\lrf{(n-i)/2}}\gamma_{n,i,j}(q)p^ix^j.$$
The first few $\gamma_{n}(x,p,q)$ are
$\gamma_{1}(x,p,q)=pq,~\gamma_{2}(x,p,q)=p^2q^2+qx,~\gamma_{3}(x,p,q)=p^3q^3+q(1+3pq)x$.
By using~\eqref{Csxz-explicit}, one can easily derive the following curious enumerative consequence.
\begin{corollary}
For $n\geqslant 2$, we have
$$\gamma_{n}(x,1,-1)=\sum_{\ell\geqslant 0}(-1)^{n-\ell}\binom{n-\ell}{\ell}x^{\ell},$$
$$\gamma_{n}(x,0,-1)=\sum_{\ell\geqslant 0}(-1)^{\ell+1}\binom{n-2-\ell}{\ell}x^{\ell+1},$$
$$\gamma_{n}(x,x,-1)=\sum_{\ell\geqslant 0}(-1)^{\ell+1}\binom{2n-2-\ell}{\ell}x^{\ell+1}.$$
\end{corollary}
%%%%%%%%%%%%%%%%%%%%%%%%%%%%%%%%%%%%%%%%%%%
\section{The symmetric decomposition of $A_{n}^{(k)}(x)$ and bi-$\gamma$-coefficients}\label{bi-gamma-coefficients}
%%%%%%%%%%%%%%%%%%%%%%%%%%%%%%%%%%%%%%%%%%%
%%%%%%%%%%%%%%%%%%%%%%%%%%%%%%%%%%%%%%%%%%%%%%%%%%%%%%%%%%%%%%%%%%%%%%%%%%%%%%%%%
%%%%%%%%%%%%%%%%%%%%%%%%%%%%%%%%%%%%%%%%%%%%%%%%%%%%%%%%%%%%%%%%%%%%%%%%%%%%%%%%%
%%%%%%%%%%%%%%%%%%%%%%%%%%%%%%%%%%%%%%%%%%%%%%%%%
In the first part of the proof of Theorem~\ref{thm001}, we get that the polynomials $A_{n}^{(k)}(x)$ are bi-$\gamma$-positive.
In this section, we shall give a combinatorial interpretation for the symmetric decomposition of $A_{n}^{(k)}(x)$
as well as bi-$\gamma$-coefficients of $A_{n}^{(2)}(x)$.
Let $A^+_{n,i;k}$ and $A^-_{n,i;k}$ be the numbers defined by~\eqref{recusystem}.
For $n\geqslant 2$, we define
\begin{align*}
{{A}}^+_{n;k}(x)=\sum_{i=0}^{\lrf{{(n-1)}/{2}}}A^+_{n,i;k}x^i,~{A}^-_{n;k}(x)=\sum_{i=0}^{\lrf{{(n-2)}/{2}}}A^-_{n,i;k}x^i.
\end{align*}
Multiplying both sides of~\eqref{recusystem} by $x^i$ and summing over all $i$, we obtain the recurrence system
\begin{equation*}\label{Stirling-recu005}
\left\{
  \begin{array}{ll}
A^+_{n+1;k}(x)&=(1+2k(n-1)x)A^+_{n;k}(x)+kx(1-4x)\frac{\mathrm{d}}{\mathrm{d}x}A^+_{n;k}(x)+xA^-_{n;k}(x),\\
A^-_{n+1;k}(x)&=k(1+2(n-2)x)A^-_{n;k}(x)+kx(1-4x)\frac{\mathrm{d}}{\mathrm{d}x}A^-_{n;k}(x)+(k-1)A^+_{n;k}(x),
  \end{array}
\right.
\end{equation*}
with $A^+_{1;k}(x)=1$ and  $A^-_{1;k}(x)=0$.
Below are the polynomials $A^+_{n;k}(x)$ and  $A^-_{n;k}(x)$ for $2\leqslant n\leqslant 4$:
\begin{align*}
A^+_{2;k}(x)&=1,~A^-_{2;k}(x)=k-1,~
A^+_{3;k}(x)=1+(3k-1)x,~A^-_{3;k}(x)=k^2-1,\\
A^+_{4;k}(x)&=1+(6k+4k^2-2)x,~A^-_{4;k}(x)=k^3-1+(1-6k+3k^2+2k^3)x.
\end{align*}
Recall that $A_n^{(k)}(x)=a_n^{(k)}(x)+xb_n^{(k)}(x)$.
It follows from~\eqref{ankxbnkx} that
\begin{align*}
a_{n}^{(k)}(x)&=(1+x)^{n-1}{{A}}^+_{n;k}\left(\frac{x}{(1+x)^2}\right),~
b_{n}^{(k)}(x)=(1+x)^{n-2}{A}^-_{n;k}\left(\frac{x}{(1+x)^2}\right).
\end{align*}

We now recall a combinatorial interpretation of $A_{n}^{(k)}(x)$.
Let $j^i=\underbrace{j,\ldots,j}_i$ for $i,j\geqslant 1$.
We say that a permutation of
$\{1^k,2^k,\ldots,n^k\}$ is a $k$-{\it Stirling permutation} of order $n$ if for each $i$, $1\leqslant i\leqslant n$,
all entries between the two occurrences of $i$ are at least $i$. When $k=2$, the $k$-Stirling permutation reduces to the classical Stirling permutation, see~\cite{Bona08,Haglund12} for instance.
Let $\mqn(k)$ be the set of $k$-{\it Stirling permutations} of order $n$.
Let $\sigma=\sigma_1\sigma_2\cdots\sigma_{nk}\in \mqn(k)$.
We say that an index $i$ is a {\it longest ascent plateau} if $\sigma_{i-1}<\sigma_i=\sigma_{i+1}=\sigma_{i+2}=\cdots=\sigma_{i+k-1}$, where $2\leqslant i\leqslant nk-k+1$.
A {\it longest left ascent plateau} of $\sigma$ is a longest ascent plateau of $\sigma$ endowed with a 0 in the front of $\sigma$.
Let $\ap(\sigma)$ (resp.~$\lap(\sigma)$) be the number of longest ascent plateaus (resp.~longest left ascent plateaus) of $\sigma$.
It is clear that
\begin{equation*}
\lap(\sigma):=\begin{cases}
\ap(\sigma)+1,& \text{if $\sigma_1=\sigma_2=\cdots =\sigma_k$};\\
\ap(\sigma), & \text{otherwise}.
\end{cases}
\end{equation*}
The following results were obtained in~\cite{Ma15}:
\begin{equation}\label{Ankx-stirling}
A_{n}^{(k)}(x)=\sum_{\sigma\in \mqn(k)}x^{\ap(\sigma)}, ~x^nA_{n}^{(k)}\left(\frac{1}{x}\right)=\sum_{\sigma\in {\mathcal{Q}}_{n}(k)}x^{\lap(\sigma)}.
\end{equation}
Note that $\deg A_{n}^{(k)}(x)=n-1$. Let $\left(a_n^{(k)}(x),b_n^{(k)}(x)\right)$ be the symmetric decomposition of $A_{n}^{(k)}(x)$.
Let $\overline{\mq}_n(k)=\{\sigma\in \mqn(k)\mid \sigma_{j}<\sigma_{j+1}{\text{ for some $j\in[k-1]$}}\}$ and let $\mqn=\mqn(2)$.
Combining~\eqref{ax-bx-prop01} and~\eqref{Ankx-stirling}, we obtain
$$a_n^{(k)}(x)=\frac{\sum_{\sigma\in \mqn(k)}x^{\ap(\sigma)}-\sum_{\sigma\in {\mathcal{Q}}_{n}(k)}x^{\lap(\sigma)}}{1-x}.$$
So we immediately get the following result.
\begin{proposition}
We have
\begin{align*}
a_n^{(k)}(x)=\sum_{\substack{\sigma \in\mqn(k)\\\sigma_{1}=\sigma_2=\cdots=\sigma_k}}x^{\ap(\sigma)},~xb_n^{(k)}(x)=\sum_{\sigma\in\overline{\mq}_n(k)}x^{\ap(\sigma)}.
\end{align*}
In particular, we have
\begin{equation*}
a_n^{(2)}(x)
=\sum_{\substack{\sigma \in\mqn\\\sigma_{1}=\sigma_2}}x^{\ap(\sigma)},~xb_n^{(2)}(x)
=\sum_{\substack{\sigma \in\mqn\\\sigma_{1}<\sigma_2}}x^{\ap(\sigma)}.
\end{equation*}
\end{proposition}
%Various statistics on Stirling permutations were repeatedly discovered, see e.g.,~\cite{Bona08,Ma19,Ma2001}.
%
%Setting $\gamma_{n,i}^+=A^+_{n,i;2}$ and $\gamma_{n,i}^-=A^-_{n,i;2}$, we have
%\begin{equation}\label{ankxbnkx}
%a_n^{(2)}(x)=\sum_{i=0}^{\lrf{{(n-1)}/{2}}}\gamma_{n,i}^+x^i(1+x)^{n-1-2i},~
%b_n^{(2)}(x)=\sum_{i=0}^{\lrf{{(n-2)}/{2}}}\gamma_{n,i}^-x^i(1+x)^{n-2-2i}.
%\end{equation}
The rest part of this section devoted to the combinatorial interpretation of bi-$\gamma$-coefficients of $A_n^{(2)}(x)$.
It follows from~\eqref{Ankx-def02} and~\eqref{Ankx-decom} that
\begin{equation*}\label{Ankx-def03}
A_n^{(2)}(x)=2^nA_n(x,1/2)=\sum_{\pi\in\msn}x^{\exc(\pi)}2^{n-\cyc(\pi)}=a_n^{(2)}(x)+xb_n^{(2)}(x).
\end{equation*}
Below are the polynomials $a_n^{(2)}(x)$ and $b_n^{(2)}(x)$ for $n\leqslant 4$:
\begin{align*}
a_1^{(2)}(x)&=1,~b_1^{(2)}(x)=0,~a_2^{(2)}(x)=1+x,~b_2^{(2)}(x)=1,~
a_3^{(2)}(x)=1+7x+x^2,\\~b_3^{(2)}(x)&=3+3x,~
a_4^{(2)}(x)=1+29x+29x^2+x^3,~b_4^{(2)}(x)=7+31x+7x^2.
%a_5^{(2)}(x)&=1+101x+321x^2+101x^3+x^4,~b_5^{(2)}(x)=15+195x+195x^2+15x^3.
\end{align*}

We say that $\pi\in\msn$ is a {\it circular permutation} if it has only one cycle. Let $A=\{x_1,\ldots, x_j\}$
be a finite set of positive integers,
and let $\cc_A$ be the set of all circular
permutations of $A$. Let $w\in\cc_A$. We will always write $w$ by using its
canonical presentation $w=y_1y_2\cdots y_j$, where $y_1=\min A, y_i=w^{i-1}(y_1)$ for $2\leqslant i\leqslant j$ and $y_1=w^j(y_1)$.
A {\it cycle peak} (resp.~{\it cycle double ascent}, {\it cycle double descent}) of $w$ is an entry $y_i$, $2\leqslant i\leqslant j$, such that
$y_{i-1}<y_i>y_{i+1}$ (resp.~$y_{i-1}<y_i<y_{i+1}$, $y_{i-1}>y_i>y_{i+1}$), where we set $y_{j+1}=+\infty$,
i.e., $y_{j+1}$ is the positive infinity. Let $\cpkk(w)$ (resp.~$\cdasc(w)$,~$\cddes(w)$) be the number of cycle peaks (resp.~cycle double ascents, cycle double descents) of $w$.
An {\it alternating run} of $w$ is a maximal consecutive subsequence that is increasing or decreasing. Following~\cite{Ma2001},
the number of {\it cycle runs} $\crun(w)$ of $w$ is defined to be the number of alternating runs of the word $y_1y_2\cdots y_j+\infty$.
Assume that $\cyc(\pi)=s$ and $\pi=w_1w_2\cdots w_s$, where $w_i$ is the $i$th cycle of $\pi$. Let $\crun(\pi)=\sum_{i=1}^s\crun(w_i)$ be
the number of cycle runs of $\pi$ and let $\cpkk(\pi)=\sum_{i=1}^s\cpkk(w_i)$ be the number of cycle peaks of $\pi$.
For $\pi\in\msn$, it is clear that $$1\leqslant \crun(\pi)\leqslant n,~\crun(\pi)=2\cpkk(\pi)+\cyc(\pi).$$
For instance, $\crun((1)(2)(3)\cdots(n))=n$ and $\crun((1,2,3,\ldots,n))=1$.
\begin{example}
If $\pi=(1,4,2)(3,5,6)(7)\in\ms_7$, then $\crun(\pi)=5$, since the numbers of cycle runs in the three cycles are $3,1,1$, respectively.
\end{example}
We can now present the third main result of this paper.
\begin{theorem}\label{thm18}
For $n\geqslant 2$, we have
$$\sum_{\pi\in\msn}x^{\exc(\pi)}2^{n-\cyc(\pi)}=\sum_{i=0}^{\lrf{{(n-1)}/{2}}}\xi_{n,i}^+x^i(1+x)^{n-1-2i}+
x\sum_{j=0}^{\lrf{{(n-2)}/{2}}}\xi_{n,j}^-x^j(1+x)^{n-2-2j}.$$
Let $\ms_{n,i}$ be the set of permutations in $\msn$ with $i$ cycle runs. Then we have
\begin{equation}\label{xi01}
\left\{
  \begin{array}{ll}
    \xi_{n,i}^+=\sum_{\pi\in\ms_{n,2i+1}}2^{\crun(\pi)-\cyc(\pi)}=\sum_{\pi\in\ms_{n,2i+1}}4^{\cpkk(\pi)}, &  \\
    \xi_{n,j}^-=\sum_{\pi\in\ms_{n,2j+2}}2^{\crun(\pi)-\cyc(\pi)}=\sum_{\pi\in\ms_{n,2j+2}}4^{\cpkk(\pi)}. &
  \end{array}
\right.
\end{equation}
%\begin{equation}\label{xi01}
%\xi_{n,i}^+=\sum_{\pi\in\ms_{n,2i+1}}2^{\crun(\pi)-\cyc(\pi)}=\sum_{\pi\in\ms_{n,2i+1}}4^{\cpkk(\pi)},~\xi_{n,j}^-=\sum_{\pi\in\ms_{n,2j+2}}2^{\crun(\pi)-\cyc(\pi)}=\sum_{\pi\in\ms_{n,2j+2}}4^{\cpkk(\pi)}.
%\end{equation}
\end{theorem}
\begin{proof}
Note that $2^2A_2(x,1/2)=1+x+x$ and $2^3A_3(x,1/2)=[(1+x)^2+5x]+3x(1+x)$.
Thus $\xi_{2,0}^+=\xi_{2,0}^-=1,~\xi_{3,0}^+=1,~\xi_{3,1}^+=5,~\xi_{3,0}^-=3$.
Note that $\ms_{2,1}=\{(12)\},~\ms_{2,2}=\{(1)(2)\}$,
$$\ms_{3,1}=\{(1,2,3)\},~\ms_{3,2}=\{(1,2)(3),(1,3)(2),(1)(2,3)\},\ms_{3,3}=\{(1)(2)(3),(1,3,2)\}.$$
It is easy to check that~\eqref{xi01} holds for $n=2,3$. We proceed by induction on $n$.
In order to get permutations in
$\ms_{n+1,2i+1}$, we distinguish three cases:
\begin{itemize}
 \item [\rm ($c_1$)] If $\pi\in\ms_{n,2i}$, then we can insert $n+1$ into $\pi$ as a new cycle. This gives the term $\xi^-_{n,i-1}$;
  \item [\rm ($c_2$)] If $\pi\in \ms_{n,2i+1}$, then $2\cpkk(\pi)+\cyc(\pi)=2i+1$.
We can insert $n+1$ just before or right after each cycle peak of $\pi$.
  Moreover, we can insert $n+1$ at the end of a cycle of $\pi$. This gives the term $(2i+1)\xi^+_{n,i}$;
    \item [\rm ($c_3$)] If $\pi\in \ms_{n,2i-1}$, then $2\cpkk(\pi)+\cyc(\pi)=2i-1$.
    We can insert $n+1$ into any of the remaining $n-(2i-1)$ positions, and the number of cycle runs is increased by two.
   This gives the term $4(n-2i+1)\xi^+_{n,i-1}$.
\end{itemize}
Similarly, there are three ways to get permutations in $\ms_{n+1,2i+2}$ by inserting the entry $n+1$:
\begin{itemize}
 \item [\rm ($c_1$)]If $\pi\in \ms_{n,2i+1}$, then we can insert $n+1$ into $\pi$ as a new cycle. This gives the term $\xi^+_{n,i}$;
  \item [\rm ($c_2$)] If $\pi\in \ms_{n,2i+2}$, then $2\cpkk(\pi)+\cyc(\pi)=2i+2$. We can insert $n+1$ just before or right after each cycle peak of $\pi$.
  Moreover, we can insert $n+1$ at the end of a cycle of $\pi$. This gives the term $(2i+2)\xi^-_{n,i}$;
    \item [\rm ($c_3$)] If $\pi\in \ms_{n,2i}$, then $2\cpkk(\pi)+\cyc(\pi)=2i$.  We can insert $n+1$ into any of the remaining $n-2i$ positions, and the number of cycle runs is increased by two. This gives the term $4(n-2i)\xi^-_{n,i-1}$.
\end{itemize}
In conclusion, we have
\begin{equation*}\label{xi}
\left\{
  \begin{array}{ll}
    \xi^+_{n+1,i}=(2i+1)\xi^+_{n,i}+4(n-2i+1)\xi^+_{n,i-1}+\xi^-_{n,i-1}, & \\
  \xi^-_{n+1,i}=(2i+2)\xi^-_{n,i}+4(n-2i)\xi^-_{n,i-1}+\xi^+_{n,i},&
  \end{array}
\right.
\end{equation*}
Comparing this with~\eqref{recusystem} leads to $\xi^+_{n,i}=A^+_{n,i;2}$ and $\xi^-_{n,i}=A^-_{n,i;2}$, and this completes the proof.
\end{proof}
%%%%%%%%%%%%%%%%%%%%%%%%%%%%%%%%%%%%%%%%%%%%%%%%%%
%%%%%%%%%%%%%%%%%%%%%%%%%%%%%%%%%%%%%%%%%%%%%%%%%%%%%%%%%%%%%%%%%%%%%%%%%%%%%%%%%%
\section{Excedances, fixed points and cycles of signed permutations}\label{section004}
%%%%%%%%%%%%%%%%%%%%%%%%%%%%%%%%%%%%%%%%%%%%
%%%%%%%%%%%%%%%%%%%%%%%%%%%%%%%%%%%%%%%%%%%%%%%%%%%%%%%%%%%%%%%%%%%%%%%%%%%%%%%%%%
%%%%%%%%%%%%%%%%%%%%%%%%%%%%%%%%%%%%%%%%%%%%%%%%%%%%%%%%%%%%%%%%%%%%%%%%%%%%%%%%%%
%%%%%%%%%%%%%%%%%%%%%%%%%%%%%%%%%%%%%%%%%%%%%%%%%%
%%%%%%%%%%%%%%%%%%%%%%%%%%%%%%%%%%%%%%%%%%%%%%%%%%%%%%%%%%%%%%%%%%%%%%%%%%%%%%%%%%
%%%%%%%%%%%%%%%%%%%%%%%%%%%%%%%%%%%%%%%%%%%%%%%%%%%%%%%%%%%%%%%%%%%%%%%%%%%%%%%%%%
\subsection{Basic definitions}
%%%%%%%%%%%%%%%%%%%%%%%%%%%%%%%%%%%%%%%%%%%%
%%%%%%%%%%%%%%%%%%%%%%%%%%%%%%%%%%%%%%%%%%%%%%%%%%%%%%%%%%%%%%%%%%%%%%%%%%%%%%%%%%
%%%%%%%%%%%%%%%%%%%%%%%%%%%%%%%%%%%%%%%%%%%%%%%%%%%%%%%%%%%%%%%%%%%%%%%%%%%%%%%%%%
%%%%%%%%%%%%%%%%%%%%%%%%%%%%%%%%%%%%%%%%%%%%%%%%%%
%%%%%%%%%%%%%%%%%%%%%%%%%%%%%%%%%%%%%%%%%%%%%%%%%%%%%%%%%%%%%%%%%%%%%%%%%%%%%%%%%%
\hspace*{\parindent}

Let $\sigma=\sigma(1)\sigma(2)\cdots\sigma(n)\in \mbn$.
It should be noted that the $n$ letters
appearing in the cycle notation of $\sigma\in \mbn$ are the
letters $\sigma(1),\sigma(2),\ldots,\sigma(n)$.
Let $\des_B(\sigma)=\#\{i\in\{0,1,\ldots,n-1\}\mid \sigma(i)>\sigma({i+1})\}$, where $\sigma(0)=0$.
We say that $i\in [n]$ is a {\it weak excedance} of $\sigma$ if $\sigma(i)=i$ or $\sigma(|\sigma(i)|)>\sigma(i)$ (see~\cite[p.~431]{Brenti94}).
Let $\we(\sigma)$ be the number of weak excedances of $\sigma$. Then $\we(\sigma)=\exc(\sigma)+\fix(\sigma)$.
According to~\cite[Theorem~3.15]{Brenti94}, the statistics $\des_B$ and $\we$ have the same distribution over $B_n$,
and their common enumerative polynomial is the
{\it type $B$ Eulerian polynomial}:
$$B_n(x)=\sum_{\sigma\in \mbn}x^{\des_B(\sigma)}=\sum_{\sigma\in \mbn}x^{\we(\sigma)}.$$

Recall that $Q(n,i)$ is the number of permutations in $\msn$ with $i$ left peaks.
By using the theory of enriched $P$-partitions, Petersen~\cite[Proposition~4.15]{Petersen07} obtained that
\begin{equation}\label{Bnxgamma}
B_n(x)=\sum_{i=0}^{\lrf{n/2}}4^iQ(n,i)x^i(1+x)^{n-2i},
\end{equation}
which has been extensively studied, see~\cite{Chow08,Lin15,Zeng16,Zhuang17} and references therein.
%Let $Q_n(x)=\sum_{i=0}^{\lrf{n/2}}4^iQ(n,i)x^i$.
%According to~\cite[Proposition~4.10]{Chow08}, we have
%\begin{equation}\label{bxz-EGF}
%Q(x;z)=1+\sum_{n=1}^\infty Q_n(x)\frac{z^n}{n!}=\frac{\sqrt{4x-1}\sec\left(z\sqrt{4x-1}\right)}
%{\sqrt{4x-1}-\tan\left(z\sqrt{4x-1}\right)}.
%\end{equation}
%%%%%%%%%%%%%%%%%%%%%%%%%%%%%%%%%%%%%%%%%%%%%%%%%%%%%%%%%%%%%%%%%%%%%%%%%%%%%%%%%%
%%%%%%%%%%%%%%%%%%%%%%%%%%%%%%%%%%%%%%%%%%%%%%%%%%
\subsection{A unified generalization of~\eqref{Bnxgamma} and Propositions~\ref{Foata70} and~\ref{Zeng12}}
%%%%%%%%%%%%%%%%%%%%%%%%%%%%%%%%%%%%%%%%%%%%
%%%%%%%%%%%%%%%%%%%%%%%%%%%%%%%%%%%%%%%%%%%%%%%%%%%%%%%%%%%%%%%%%%%%%%%%%%%%%%%%%%
%%%%%%%%%%%%%%%%%%%%%%%%%%%%%%%%%%%%%%%%%%%%%%%%%%%%%%%%%%%%%%%%%%%%%%%%%%%%%%%%%%
%%%%%%%%%%%%%%%%%%%%%%%%%%%%%%%%%%%%%%%%%%%%%%%%%%
%%%%%%%%%%%%%%%%%%%%%%%%%%%%%%%%%%%%%%%%%%%%%%%%%%%%%%%%%%%%%%%%%%%%%%%%%%%%%%%%%%
\hspace*{\parindent}

We say that $i$ is an {\it excedance} (resp.~{\it anti-excedance}, {\it fixed point}, {\it singleton}) of $\sigma$ if $\sigma(|\sigma(i)|)>\sigma(i)$ (resp.~$\sigma(|\sigma(i)|)<\sigma(i)$, $\sigma(i)=i$, $\sigma(i)=\overline{i}$).
Let $\exc(\sigma)$ (resp.~$\aexc(\sigma)$, $\fix(\sigma)$, $\single(\sigma)$, $\negg(\sigma)$) be the number of excedances (resp.~anti-excedances, fixed points, singletons, negative entries) of $\sigma$.
\begin{example}
The signed permutation $\pi=\overline{3}51\overline{7}2468\overline{9}$ can be written as
$(\overline{9})(\overline{3},1)(2,5)(4,\overline{7},6)(8)$. Moreover, $\pi$ with only one
singleton $9$ and one fixed point $8$, and $\pi$ has $3$ excedances, $4$ anti-excedances and $3$ negative entries.
\end{example}

Consider the following polynomials
$$B_n(x,y,s,t,p,q)=\sum_{\sigma\in \mbn}x^{\exc(\sigma)}y^{\aexc(\sigma)}s^{\single(\sigma)}t^{\fix(\sigma)}p^{\negg(\sigma)}q^{\cyc(\sigma)}.$$
Clearly, $A_n(x,q)=B_n(x,1,1,1,0,q)$ and $B_n(x)=B_n(x,1,1,x,1,1)$.
%Below are the polynomials $B_n(x,y,s,t,p,q)$ for $0\leqslant n\leqslant 3$:
%\begin{align*}
%B_0(x,y,s,t,p,q)&=1,\\
%B_1(x,y,s,t,p,q)&=q(s+pt),\\
%B_2(x,y,s,t,p,q)&=q^2(s+pt)^2+q(1+p)^2xy,\\
%B_3(x,y,s,t,p,q)&=q^3(s+pt)^3+3q^2(1+p)^2(s+pt)xy+q(1+p)^3xy(x+y).
%\end{align*}
%Let $$B(x,y,s,t,p,q;z)=\sum_{n=0}^\infty B_n(x,y,s,t,p,q)\frac{z^n}{n!}.$$
We can now present the fourth main result of this paper.
\begin{theorem}\label{mainthm004}
We have
\begin{equation}\label{Bxy-EGF}
B_n(x,y,s,t,p,q)=(1+p)^ny^nA_n\left(\frac{x}{y},\frac{t+sp}{(1+p)y},q\right).
\end{equation}
\end{theorem}

%Define
%$$d_n^B(x,p)=B_n(x,1,1,0,p,1)=\sum_{\sigma\in \md_n^B}x^{\exc(\sigma)}p^{\negg(\sigma)}.$$
%%It should be noted that the polynomials $d_n^B(x,1)$ have been studied by Chow~\cite{Chow09} and the polynomials
%%$x^nd_n^B(1/x,1)$ have been studied by Chen et al.~\cite{Chen09} and Shin and Zeng~\cite{Zeng16}.
%Note that $d_n^B(x,1)=d_n^B(x)$ and $ d_n^B(x,0)=d_n(x)$.
%Below are the polynomials $d_n^B(x,q)$ for $0\leq n\leq 3$:
%\begin{align*}
%d_0^B(x,q)&=1,~
%d_1^B(x,q)=q,~
%d_2^B(x,q)=x+2qx+q^2(1+x),\\
%d_3^B(x,q)&=x(1+x)+3qx(2+x)+3q^2x(3+x)+q^3(1+4x+x^2),\\
%d_4^B(x,q)&=x(1+7x+x^2)+4qx(2+8x+x^2)+6q^2x(4+9x+x^2)+\\&4q^3x(7+10x+x^2)+q^4(1+11x+11x^2+x^3).
%\end{align*}
In the rest part of this subsection, we shall prove Theorem~\ref{mainthm004}.
\begin{lemma}\label{lemmaBn}
If $V=\{J,s,t,x,y\}$ and
\begin{equation}\label{Ixyst-G}
G_3=\{J\rightarrow qJ(t+sp),s\rightarrow (1+p)xy,t\rightarrow (1+p)xy,x\rightarrow (1+p)xy,y\rightarrow (1+p)xy\},
\end{equation}
then we have $D_{G_3}^n(J)=JB_n(x,y,s,t,p,q)$.
\end{lemma}
\begin{proof}
We first introduce a grammatical labeling of $\sigma\in \mbn$ as follows:
\begin{itemize}
  \item [\rm ($L_1$)]If $i$ is an excedance, then put a superscript label $x$ right after $\sigma(i)$;
 \item [\rm ($L_2$)]If $i$ is a anti-excedance, then put a superscript label $y$ right after $\sigma(i)$;
\item [\rm ($L_3$)]If $i$ is a fixed point, then put a superscript label $t$ right after $i$;
\item [\rm ($L_4$)]If $i$ is a singleton, then put a superscript label $s$ right after $\overline{i}$;
\item [\rm ($L_5$)]Put a superscript label $J$ at the end of $\sigma$;
\item [\rm ($L_6$)]Put a subscript label $q$ at the end of each cycle of $\sigma$;
\item [\rm ($L_7$)]Put a subscript label $p$ right after each negative entry of $\sigma$.
\end{itemize}
For example, let $\sigma=(1,3,\overline{2},6)(\overline{4})(5)$.
The grammatical labeling of $\sigma$ is given below:
$$(1^x3^y\overline{2}^x_p6^y)_q(\overline{4}^s_p)_q(5^t)_q^J.$$
Note that the weight of $\sigma$ is given by $w(\sigma)=Jx^{\exc(\sigma)}y^{\aexc(\sigma)}s^{\single(\sigma)}t^{\fix(\sigma)}p^{\negg(\sigma)}q^{\cyc(\sigma)}$.
For $n=1$, we have $\mb_1=\{(1^t)^J_q,(\overline{1}^s_p)^J_q\}$.
Note that $D_{G_3}(J)=qJ(t+sp)$.
Hence the result holds for $n=1$. We proceed by induction.
Suppose we get all labeled permutations in $\sigma\in \mb_{n-1}$, where $n\geqslant 2$. Let
$\widehat{{\sigma}}$ be obtained from $\sigma\in \mb_{n-1}$ by inserting the entry $n$ or $\overline{n}$.
There are five cases to label the inserted element and relabel some elements of $\sigma$:
\begin{itemize}
  \item [\rm ($c_1$)]If $n$ or $\overline{n}$ appear a new cycle, then the changes of labeling are illustrated as follows:
$$\cdots(\cdots)_q^J\rightarrow\cdots(\cdots)_q(n^t)_q^J,\quad \cdots(\cdots)^J\rightarrow\cdots(\cdots)_q(\overline{n}_p^s)_q^J;$$
 \item [\rm ($c_2$)]If we insert $n$ or $\overline{n}$ right after a fixed point, then the changes of labeling are illustrated as follows:
$$\cdots(i^t)_q(\cdots)\cdots\rightarrow\cdots(i^xn^y)_q(\cdots)\cdots,\quad \cdots(i^t)_q(\cdots)\cdots\rightarrow\cdots(i^y\overline{n}_p^x)_q(\cdots)\cdots;$$
\item [\rm ($c_3$)]If we insert $n$ or $\overline{n}$ right after a singleton, then the changes of labeling are illustrated as follows:
$$\cdots(\overline{i}^s_p)_q(\cdots)\cdots\rightarrow\cdots(\overline{i}^x_pn^y)_q(\cdots)\cdots,\quad \cdots(\overline{i}^s_p)_q(\cdots)\cdots\rightarrow\cdots(\overline{i}^y_p\overline{n}_p^x)_q(\cdots)\cdots;$$
\item [\rm ($c_4$)]If we insert $n$ or $\overline{n}$ right after an excedance, then the changes of labeling are illustrated as follows:
$$\cdots(\cdots \sigma(i)^x\sigma(|\sigma(i)|)\cdots)_q(\cdots)\cdots\rightarrow\cdots(\cdots \sigma(i)^xn^y\sigma(|\sigma(i)|)\cdots)_q(\cdots)\cdots,$$
$$\cdots(\cdots \sigma(i)^x\sigma(|\sigma(i)|)\cdots)_q(\cdots)\cdots\rightarrow\cdots(\cdots \sigma(i)^y\overline{n}_p^x\sigma(|\sigma(i)|)\cdots)_q(\cdots)\cdots;$$
\item [\rm ($c_5$)]If we insert $n$ or $\overline{n}$ right after an anti-excedance, then the changes of labeling are illustrated as follows:
$$\cdots(\cdots \sigma(i)^y\sigma(|\sigma(i)|)\cdots)_q(\cdots)\cdots\rightarrow\cdots(\cdots \sigma(i)^xn^y\sigma(|\sigma(i)|)\cdots)_q(\cdots)\cdots,$$
$$\cdots(\cdots \sigma(i)^y\sigma(|\sigma(i)|)\cdots)_q(\cdots)\cdots\rightarrow\cdots(\cdots \sigma(i)^y\overline{n}_p^x\sigma(|\sigma(i)|)\cdots)_q(\cdots)\cdots.$$
\end{itemize}
In each case, the insertion of $n$ or $\overline{n}$ corresponds to one substitution rule in $G$. By induction,
it is routine to check that the action of $D_{G_3}$
on the set of weights of signed permutations in $\mb_{n-1}$ gives the set of weights of signed permutations in $\mb_{n}$.
This completes the proof.
\end{proof}

\noindent{\bf A proof of~\eqref{Bxy-EGF}:}
\begin{proof}
Let $G_3$ be the grammar given in Lemma~\ref{lemmaBn}.
Consider a change of the grammar $G_3$. Setting $A=t+sp,B=(1+p)x$ and $C=(1+p)y$, we get
$$D_{G_3}(J)=qJA,~D_{G_3}(A)=BC,~D_{G_3}(B)=BC,~D_{G_3}(C)=BC.$$
Let $G_4=\{J\rightarrow qJA, A\rightarrow BC, B\rightarrow BC, C\rightarrow BC\}$.
It follows from Lemma~\ref{lemma001exc} that
\begin{equation*}
D_{G_4}^n(J)=J\sum_{\pi\in\msn}A^{\fix(\pi)}B^{\exc(\pi)}C^{\drop(\pi)}q^{\cyc(\pi)}.
\end{equation*}
Then upon taking $A=t+sp,B=(1+p)x$ and $C=(1+p)y$ in the above expansion, we get
\begin{equation}\label{DG3}
D_{G_3}^n(J)=J\sum_{\pi\in\msn}(t+sp)^{\fix(\pi)}(x+px)^{\exc(\pi)}(y+py)^{\drop(\pi)}q^{\cyc(\pi)}.
\end{equation}
Combining~\eqref{DG3} and Lemma~\ref{lemmaBn}, we obtain~\eqref{Bxy-EGF}.
This completes the proof.
\end{proof}
%%%%%%%%%%%%%%%%%%%%%%%%%%%%%%%%%%%%%%%%%%%%%%%%%%%%%%%%%%%%%%%%%%%%%%%%%%%%%%%%%%
\subsection{Some applications of Theorem~\ref{mainthm004}}
%%%%%%%%%%%%%%%%%%%%%%%%%%%%%%%%%%%%%%%%%%%%
%%%%%%%%%%%%%%%%%%%%%%%%%%%%%%%%%%%%%%%%%%%%%%%%%%%%%%%%%%%%%%%%%%%%%%%%%%%%%%%%%%
%%%%%%%%%%%%%%%%%%%%%%%%%%%%%%%%%%%%%%%%%%%%%%%%%%%%%%%%%%%%%%%%%%%%%%%%%%%%%%%%%%
%%%%%%%%%%%%%%%%%%%%%%%%%%%%%%%%%%%%%%%%%%%%%%%%%%
%%%%%%%%%%%%%%%%%%%%%%%%%%%%%%%%%%%%%%%%%%%%%%%%%%%%%%%%%%%%%%%%%%%%%%%%%%%%%%%%%%
\hspace*{\parindent}

Combining~\eqref{eq:proof03} and~\eqref{Bxy-EGF}, we immediately get the following result.
\begin{corollary}
Let $\gamma_{n,i,j}(q)$ be the polynomials defined by~\eqref{bnij-combin}.
Then we have
\begin{equation*}\label{Bnxyspq-gamma}
B_n(x,y,s,t,p,q)=\sum_{i=0}^n(t+sp)^i(1+p)^{n-i}\sum_{j=0}^{\lrf{(n-i)/2}}\gamma_{n,i,j}(q)(xy)^j(x+y)^{n-i-2j}.
\end{equation*}
In particular, setting $y=s=1$ and $t=0$ in $B_n(x,y,s,t,p,q)$, we get
\begin{equation}\label{Dnbxq01}
\sum_{\sigma\in \md_n^B}x^{\exc(\sigma)}p^{\negg(\sigma)}q^{\cyc(\pi)}=\sum_{i=0}^np^i(1+p)^{n-i}\sum_{j=0}^{\lrf{(n-i)/2}}\gamma_{n,i,j}(q)x^j(x+1)^{n-i-2j}.
\end{equation}
\end{corollary}
When $p=0$, then~\eqref{Dnbxq01} reduces to~\eqref{dnxq-def}.
From~\eqref{Bxy-EGF}, we get
\begin{equation}\label{Bnx1st1q}
B_n(x,1,s,t,1,q)=\sum_{\sigma\in\mbn}x^{\exc(\pi)}s^{\single(\sigma)}t^{\fix(\pi)}q^{\cyc(\pi)}=2^nA_n\left(x,\frac{t+s}{2},q\right).
\end{equation}
It should be note that $$B_n(x,1,1,-1,1,1)=2^nd_n(x),~B_n(x,1,1,0,1,1)=d_n^B(x),~B_n(x,1,2,0,1,1)=2^nA_n(x).$$
Combining Theorem~\ref{thm001} and~\eqref{Bnx1st1q}, we get the following result.
\begin{corollary}
Let $s,t$ and $q$ be three given real numbers satisfying $0\leqslant s+t\leqslant 2$ and $0\leqslant q\leqslant 1$. Then
$B_n(x,1,s,t,1,q)$ are alternatingly increasing for $n\geqslant 1$. In particular, when $t+s=2$, the polynomials $B_n(x,1,s,t,1,q)$
are bi-$\gamma$-positive.
\end{corollary}

In the rest part of this subsection, we shall
consider the following multivariate polynomials
$$B_n(x,y,s,t,p,1)=\sum_{\sigma\in \mb_n}x^{\exc(\sigma)}y^{\aexc(\sigma)}s^{\single(\sigma)}t^{\fix(\sigma)}p^{\negg(\sigma)}.$$
Set $B_0(x,y,s,t,p,1)=1$.
If follows from~\eqref{Anxpq-EGF} and~\eqref{Bxy-EGF} that $$B(x,y,s,t,p,1;z)=\sum_{n=0}^\infty B_n(x,y,s,t,p,1)\frac{z^n}{n!}=\frac{(y-x)\mathrm{e}^{(t+sp)z}}{y\mathrm{e}^{(1+p)xz}-x\mathrm{e}^{(1+p)yz}}.$$
%The first few $B_n(x,y,s,t,p)$ are given as follows:
%\begin{align*}
%B_1(x,y,s,t,p)&=t+sp,\\
%B_2(x,y,s,t,p)&=(t+sp)^2+(1+p)^2xy,\\
%B_3(x,y,s,t,p)&=(t+sp)^3+3(1+p)^2(t+sp)xy+(1+p)^3xy(x+y).
%\end{align*}
We define
$$\Phi_{n}(x,y)=xy\frac{x^{n-1}-y^{n-1}}{x-y}=xy(x^{n-2}+x^{n-3}y+\cdots+xy^{n-3}+y^{n-2}) ~~\text{for $n\geqslant 2$}.$$
In particular, $\Phi_{2}(x,y)=xy$ and $\Phi_{3}(x,y)=xy(x+y)$. Set $\Phi_{0}(x,y)=\Phi_{1}(x,y)=0$.
%We now present the following result.
\begin{theorem}\label{Bnxystq-recu}
For $n\geqslant 2$, we have
\begin{equation}\label{Bnxystq}
B_n(x,y,s,t,p,1)=(t+sp)^n+\sum_{k=0}^{n-2}\binom{n}{k}B_k(x,y,s,t,p,1)\Phi_{n-k}(x,y)(1+p)^{n-k}.
\end{equation}
\end{theorem}
\begin{proof}
It is easy to verify that
$$\Phi(x,y,p;z):=\sum_{n=0}^\infty\Phi_{n}(x,y)(1+p)^n\frac{z^n}{n!}=1-\frac{y\mathrm{e}^{(1+p)xz}-x\mathrm{e}^{(1+p)yz}}{y-x}.$$
For $n\geqslant 2$, we define
\begin{equation}\label{fnxystq}
f_n(x,y,s,t,p)=(t+sp)^n+\sum_{k=0}^{n-2}\binom{n}{k}B_k(x,y,s,t,p,1)\Phi_{n-k}(x,y)(1+p)^{n-k}.
\end{equation}
Set $f_0(x,y,s,t,p)=1$ and $f_1(x,y,s,t,p)=t+sp$.
It follows from~\eqref{fnxystq} that
\begin{align*}
f(x,y,s,t,p;z)&=\sum_{n=0}^\infty f_n(x,y,s,t,p)\frac{z^n}{n!}\\
=&\mathrm{e}^{(t+sp)z}+B(x,y,s,t,p,1;z)\Phi(x,y,p;z)\\
&=\mathrm{e}^{(t+sp)z}+\left(\frac{(y-x)\mathrm{e}^{(t+sp)z}}{y\mathrm{e}^{(1+p)xz}-x\mathrm{e}^{(1+p)yz}}\right)\left(1-\frac{y\mathrm{e}^{(1+p)xz}-x\mathrm{e}^{(1+p)yz}}{y-x}\right)\\
&=B(x,y,s,t,p,1;z),
\end{align*}
which leads to the desired result.
This completes the proof.
\end{proof}

Note that $A_n(x)=B_n(x,1,0,1,0,1),~d_n(x)=B_n(x,1,0,0,0,1)$,
$B_n(x)=B_n(x,1,1,x,1,1)$ and $d_n^B(x)=B_n(x,1,1,0,1,1)$.
The following corollary is immediate from Theorem~\ref{Bnxystq-recu} by considering some special cases.
\begin{corollary}\label{coro}
For $n\geqslant 2$, we have
\begin{align*}
A_n(x)&=1+\sum_{k=0}^{n-2}\binom{n}{k}A_k(x)(x+x^2+\cdots+x^{n-1-k});\\
d_n(x)&=\sum_{k=0}^{n-2}\binom{n}{k}d_k(x)(x+x^2+\cdots+x^{n-1-k});\\
B_n(x)&=(1+x)^n+\sum_{k=0}^{n-2}\binom{n}{k}B_k(x)(x+x^2+\cdots+x^{n-1-k})2^{n-k};\\
d_n^B(x)&=1+\sum_{k=0}^{n-2}\binom{n}{k}d_k^B(x)(x+x^2+\cdots+x^{n-1-k})2^{n-k}.
\end{align*}
\end{corollary}
Very recently, by using the theory of geometric combinatorics, Juhnke-Kubitzke, Murai and Sieg~\cite[Corollary~4.2]{Sieg19}
obtained the recurrence relation of $d_n(x)$ that is given in Corollary~\ref{coro}.
%%%%%%%%%%%%%%%%%%%%%%%%%%%%%%%%%%%%%%%%%%%%%%%%%%%%%%%%%%%%%%%%%%%%%%%%%%%%%%%%%%
%%%%%%%%%%%%%%%%%%%%%%%%%%%%%%%%%%%%%%%%%%%%%%%%%%%%%%%%%%%%%%%%%%%%%%%%%%%%%%%%%%
%%%%%%%%%%%%%%%%%%%%%%%%%%%%%%%%%%%%%%%%%%%%%%%%%%
\subsection{The flag excedance statistic of signed permutations}
%%%%%%%%%%%%%%%%%%%%%%%%%%%%%%%%%%%%%%%%%%%%
%%%%%%%%%%%%%%%%%%%%%%%%%%%%%%%%%%%%%%%%%%%%%%%%%%%%%%%%%%%%%%%%%%%%%%%%%%%%%%%%%%
%%%%%%%%%%%%%%%%%%%%%%%%%%%%%%%%%%%%%%%%%%%%%%%%%%%%%%%%%%%%%%%%%%%%%%%%%%%%%%%%%%
%%%%%%%%%%%%%%%%%%%%%%%%%%%%%%%%%%%%%%%%%%%%%%%%%%
%%%%%%%%%%%%%%%%%%%%%%%%%%%%%%%%%%%%%%%%%%%%%%%%%%%%%%%%%%%%%%%%%%%%%%%%%%%%%%%%%%
\hspace*{\parindent}

We say that $i\in[n]$ is an index of {\it type $A$ excedance} if $\sigma(i)>i$.
For $\sigma\in\mbn$,
we let
\begin{align*}
\exc_A(\sigma)&=\#\{i\in [n]: \sigma(i)>i\},~\negg(\sigma)=\#\{|\sigma(i)|\in [n]: \sigma(i)<0\},\\
\fix(\sigma)&=\#\{i\in [n]: \sigma(i)=i\},~\single(\sigma)=\#\{i\in [n]: \sigma(i)=\overline{i}\},\\
\fexc(\sigma)&=2\exc_A(\sigma)+\negg(\sigma),~\aexc_A(\sigma)=n-\exc_A(\sigma)-\fix(\sigma)-\single(\sigma).
\end{align*}

\begin{example}
The standard cycle decomposition of $\sigma=2~\overline{5}~1~3~4~\overline{6}~8~7$ is $(1,2,\overline{5},4,3)(\overline{6})(7,8)$.
Moreover, we have $\exc_A(\sigma)=\#\{1,7\}=2,~\negg(\sigma)=\#\{5,6\}=2$ and $\fexc(\sigma)=6$.
\end{example}

The flag excedance statistic $\fexc$ has been studied by Bagno and Garber~\cite{Bagno04}, Foata and Han~\cite{Foata11}, Mongelli~\cite{Mongelli13}, Shin and Zeng~\cite{Zeng16} and Zhuang~\cite{Zhuang17}. In particular,
Mongelli~\cite[Section~3]{Mongelli13} found two formulas:
\begin{equation}\label{Mongelli01}
\sum_{\sigma\in\mbn}x^{2\exc_A(\sigma)}p^{\negg(\sigma)}=(1+p)^nA_n\left(\frac{x^2+p}{1+p}\right),
\end{equation}
\begin{equation}\label{Mongelli02}
\sum_{\sigma\in\mdn^B}x^{2\exc_A(\sigma)}p^{\negg(\sigma)}=\sum_{k=0}^n\binom{n}{k}(1+p)^kp^{n-k}d_k\left(\frac{x^2+p}{1+p}\right).
\end{equation}
Setting $q=x$ in~\eqref{Mongelli01} leads to the well known formula (see~\cite{Adin2001,Foata11} for instance):
\begin{equation*}\label{mbnAnx}
\sum_{\sigma\in\mbn}x^{\fexc(\sigma)}=(1+x)^nA_n(x).
\end{equation*}
Let $$D_n^B(x)=\sum_{\sigma\in\mdn^B}x^{\fexc(\sigma)}$$ be the {\it flag derangement polynomials}.
It follows from~\eqref{Mongelli02} that
\begin{equation*}
D_n^B(x)=\sum_{k=0}^n\binom{n}{k}(1+x)^kx^{n-k}d_k(x).
\end{equation*}
By using the above expansion, Mongelli~\cite[Proposition~3.5]{Mongelli13} proved the symmetry of $D_n^B(x)$. Subsequently,
Shin and Zeng~\cite[Corollary~5]{Zeng16} proved that $D_n^B(x)$ is $\gamma$-positive.

Let
$$B_n^{(A)}(x,y,s,t,p,q)=\sum_{\sigma\in\mbn}x^{\exc_A(\sigma)}y^{\aexc_A(\sigma)}s^{\single(\sigma)}t^{\fix(\sigma)}p^{\negg(\sigma)}q^{\cyc(\sigma)}.$$
Now we give the fifth main result.
\begin{theorem}\label{thm002}
One has
\begin{equation}\label{Bnxqy01}
B_n^{(A)}(x,y,s,t,p,q)=\sum_{\pi\in\msn}(x+py)^{\exc(\pi)}(y+py)^{\drop(\pi)}(t+sp)^{\fix(\pi)}q^{\cyc(\sigma)}.
\end{equation}
Equivalently,
\begin{equation}\label{Bnxqy}
B_n^{(A)}(x,y,s,t,p,q)=(1+p)^ny^nA_n\left(\frac{x+py}{y+py},\frac{t+sp}{y+py},q\right).
\end{equation}
\end{theorem}
\begin{proof}
We claim that if $G_6=\{I\rightarrow Iq(t+sp),~t\rightarrow y(x+py),~s\rightarrow y(x+py),~x\rightarrow y(x+py),~y\rightarrow y(x+py)\}$.
Then we have
\begin{equation}\label{DnI-grammar}
D_{G_6}^n(I)=I\sum_{\sigma\in\mbn}x^{\exc_A(\sigma)}y^{\aexc_A(\sigma)}s^{\single(\sigma)}t^{\fix(\sigma)}p^{\negg(\sigma)}q^{\cyc(\sigma)}.
\end{equation}
Let $\sigma\in \mbn$.
We introduce a grammatical labeling of $\sigma$ as follows:
\begin{itemize}
 \item [\rm ($L_1$)] Put a subscript label $p$ right after every negative element of $\sigma$;
  \item [\rm ($L_2$)]If $i$ is a fixed point, then put a superscript label $t$ right after $i$, i.e., $(i^t)$;
    \item [\rm ($L_3$)]If $i$ is a singleton, then put a superscript label $s$ right after $\overline{i}$, i.e., $(\overline{i}^s)$;
 \item [\rm ($L_4$)]If $\sigma(i)>i$, then put a superscript label $x$ just before $\sigma(i)$;
  \item [\rm ($L_5$)]If $\sigma(i)<i$, then put a superscript label $y$ right after $i$ or $\overline{i}$;
\item [\rm ($L_6$)] Put a subscript label $q$ right after each cycle;
\item [\rm ($L_7$)]Put a superscript label $I$ right after $\sigma$.
\end{itemize}
For example, if $\sigma=(1,2,\overline{5},4,3)(\overline{6})(7,8)(9,\overline{10})(\overline{11},12)$, then
the grammatical labeling of $\sigma$ is given as follows:
$$(1^x2^y\overline{5}_p^y4^y3^y)_q(\overline{6}_p^s)_q(7^x8^y)_q(9^y\overline{10}_p^y)_q(\overline{11}_p^x12^y)_q^I.$$
Note that the weight of $\sigma$ is given by $Ix^{\exc_A(\sigma)}y^{\aexc_A(\sigma)}s^{\single(\sigma)}t^{\fix(\sigma)}p^{\negg(\sigma)}q^{\cyc(\sigma)}$.
Every permutation in $\mbn$ can be obtained from a permutation in $\mb_{n-1}$ by inserting $n$ or $\overline{n}$.
For $n=1$, we have $\mb_1=\{(1^t)^I_q,~(\overline{1}^s_p)^I_q\}$.
Note that $D_{G_6}(I)=Iq(t+sp)$.
Then the sum of weights of the elements in $\mb_1$ is given by $D_{G_6}(I)$.
Hence the result holds for $n=1$.

We proceed by induction on $n$.
Suppose that we get all labeled permutations in $\mb_{n-1}$, where $n\geqslant 2$. Let
$\widetilde{\sigma}$ be obtained from $\sigma\in \mb_{n-1}$ by inserting $n$ or $\overline{n}$.
When the inserted $n$ or $\overline{n}$ forms a new cycle, the insertion corresponds to the substitution rule $I\rightarrow Iq(t+sp)$.
Now we insert $n$ or $\overline{n}$ right after $\sigma(i)$.
If $i$ is a fixed point or a singleton of $\sigma$, then the changes of labeling are respectively illustrated as follows:
$$\cdots(i^t)_q\cdots\rightarrow \cdots(i^xn^y)_q\cdots,~\cdots(i^t)_q\cdots\rightarrow \cdots({i}^y\overline{n}_p^y)_q\cdots;$$
$$\cdots(\overline{i}_p^s)_q\cdots\rightarrow \cdots(\overline{i}_p^xn^y)_q\cdots,~\cdots(\overline{i}_p^s)_q\cdots\rightarrow \cdots(\overline{i}_p^y\overline{n}_p^y)_q\cdots.$$
If $i$ is an excedance of type $A$, then the changes of labeling are respectively illustrated as follows:
$$\cdots(\cdots \sigma(i)^x\sigma(|\sigma(i)|)\cdots)_q\cdots\rightarrow \cdots(\cdots \sigma(i)^xn^y\sigma(|\sigma(i)|)\cdots)_q\cdots,$$
$$\cdots(\cdots \sigma(i)^x\sigma(|\sigma(i)|)\cdots)_q\cdots\rightarrow \cdots(\cdots \sigma(i)^y\overline{n}_p^y\sigma(|\sigma(i)|)\cdots)_q\cdots.$$
The same argument applies to the case when the new element
is inserted right after an element labeled by $y$.
In each case, the insertion of $n$ or $\overline{n}$ corresponds to one substitution rule in $G_6$.
By induction, it is routine to check that the action of $D_{G_6}$ on the set of weights of signed permutations
in $\mb_{n-1}$ gives the set of weights of signed permutations in $\mb_n$, and so~\eqref{DnI-grammar} holds.

For the grammar $G_6$, setting $A=t+sp,B=x+py$ and $C=(1+p)y$, we get
$$D_{G_6}(I)=IqA,~D_{G_6}(A)=BC,~D_{G_6}(B)=BC,~D_{G_6}(C)=BC.$$
Let $G_7=\{I\rightarrow IqA, A\rightarrow BC, B\rightarrow BC, C\rightarrow BC\}$.
It follows from Lemma~\ref{lemma001exc} that
\begin{equation}\label{DG7}
D_{G_7}^n(I)=I\sum_{\pi\in\msn}A^{\fix(\pi)}B^{\exc(\pi)}C^{\drop(\pi)}q^{\cyc(\pi)}.
\end{equation}
Then upon taking $A=t+sp,B=x+py$ and $C=(1+p)y$ in~\eqref{DG7}, we get~\eqref{Bnxqy01}. This completes the proof.
\end{proof}

Combining~\eqref{Anxpq-EGF} and~\eqref{Bnxqy}, we immediately get
\begin{equation}\label{Bnxqy-EGF}
\sum_{n=0}^\infty B_n^{(A)}(x,y,s,t,p,q)\frac{z^n}{n!}=\left(\frac{(y-x)\mathrm{e}^{(t+sp)z}}{(1+p)y\mathrm{e}^{(x+py)z}-(x+py)\mathrm{e}^{(1+p)yz}}\right)^q.
\end{equation}

Replacing $x$ by $x^2$ and setting $y=t=s=q=1$ in~\eqref{Bnxqy}, we immediately get~\eqref{Mongelli01}.
Replacing $x$ by $x^2$ and setting $y=s=q=1$ and $t=0$ in~\eqref{Bnxqy}, one can easily derive~\eqref{Mongelli02}.
Note that $$B_n^{(A)}(x^2,1,1,0,px,q)=\sum_{\sigma\in \mdn^B}x^{\fexc(\sigma)}p^{\negg(\sigma)}q^{\cyc(\sigma)}.$$
It follows from~\eqref{Bnxqy-EGF} that
\begin{equation}\label{Dnxpq-EGF}
\sum_{n=0}^\infty B_n^{(A)}(x^2,1,1,0,px,q)\frac{z^n}{n!}=\left(\frac{(1-x^2)\mathrm{e}^{pxz}}{(1+px)\mathrm{e}^{(x^2+px)z}-(x^2+px)\mathrm{e}^{(1+px)z}}\right)^q.
\end{equation}
By using~\eqref{Dnxpq-EGF}, one can easily derive the following curious enumerative consequence.
\begin{corollary}
For $n\geqslant 1$, we have
$$\sum_{\sigma\in \mdn^B}x^{\fexc(\sigma)}p^{\negg(\sigma)}(-1)^{\cyc(\sigma)}=-\sum_{i=1}^{n-1}x^{2i}-p\sum_{i=1}^{n}x^{2i-1}.$$
\end{corollary}

It would be interesting to find a direct combinatorial proof of the preceding
corollary.
In the rest part of this subsection, we consider the following bivariate polynomials:
\begin{align*}
F^{(\fexc,\cyc)}_n(x,q)&=\sum_{\sigma\in\mbn}x^{\fexc(\sigma)}q^{\cyc(\sigma)},\\
D^{(\fexc,\cyc)}(x,q)&=\sum_{\sigma\in\mdn^B}x^{\fexc(\sigma)}q^{\cyc(\sigma)},\\
F^{(\fexc,\negg)}(x,p)&=\sum_{\sigma\in\mbn}x^{\fexc(\sigma)}p^{\negg(\sigma)}.
\end{align*}

From~\eqref{Bnxqy}, we see that $F^{(\fexc,\cyc)}_n(x,q)=B_n^{(A)}(x^2,1,1,1,x,q)=(1+x)^nA_n(x,q)$.
Therefore, by using Lemma~\ref{lemma01} and Theorem~\ref{thm001}, we get the following corollary.
\begin{corollary}
Let $0\leqslant q\leqslant 1$ be a given real number. For $n\geqslant 1$, the polynomials
$F^{(\fexc,\cyc)}_n(x,q)$ are bi-$\gamma$-positive, and so they are alternatingly increasing.
\end{corollary}

By using~\eqref{Bnxqy}, we get $$D^{(\fexc,\cyc)}_n(x,q)=B_n^{(A)}(x^2,1,1,0,x,q)=(1+x)^nA_n\left(x,\frac{x}{1+x},q\right).$$
Therefore,
\begin{align*}
D^{(\fexc,\cyc)}_n(x,q)&=\sum_{\pi\in\msn}x^{\exc(\pi)}x^{\fix(\pi)}(1+x)^{n-\fix(\pi)}q^{\cyc(\pi)}\\
&=\sum_{i=0}^n\binom{n}{i}(qx)^i(1+x)^{n-i}\sum_{\pi\in\md_{n-i}}x^{\exc(\pi)}q^{\cyc(\pi)}.
\end{align*}
Let $q>0$ be a given real number.
Combining Proposition~\ref{Zeng12} and Lemma~\ref{lemma01}, we find that for $0\leqslant i\leqslant n$, the
polynomials
$(qx)^i(1+x)^{n-i}\sum_{\pi\in\md_{n-i}}x^{\exc(\pi)}q^{\cyc(\pi)}$
are all $\gamma$-positive with the same centre of symmetry. So we get the following result, which is a generalization of~\cite[Corollary~5]{Zeng16}.
\begin{corollary}\label{corfexc}
Let $q>0$ be a given real number. Then the polynomials $D^{(\fexc,\cyc)}_n(x,q)$ are $\gamma$-positive for $n\geqslant 1$.
\end{corollary}

It follows from~\eqref{Bnxqy} that
\begin{equation}\label{Enxp}
F^{(\fexc,\negg)}(x,p)=B_n^{(A)}(x^2,1,1,1,px,1)=(1+px)^nA_n\left(\frac{x^2+px}{1+px}\right).
\end{equation}
In particular, $F^{(\fexc,\negg)}(x,-1)=(1-x)^nA_n(-x)$.
We can now present the following result.
\begin{theorem}
Let $p\geqslant1$ be a given real number. For $n\geqslant 1$, the polynomials $F^{(\fexc,\negg)}(x,p)$ are bi-$\gamma$-positive, and so they are alternatingly increasing.
\end{theorem}
\begin{proof}
Combining~\eqref{Enxp} and Proposition~\ref{Foata70}, we obtain $$F^{(\fexc,\negg)}(x,p)=(1+xp)\widetilde{E}_n(x,p)=(1+x)\widetilde{E}_n(x,p)+(p-1)x\widetilde{E}_n(x,p)$$
for $n\geqslant1$, where
$$\widetilde{E}_n(x,p)=\sum_{i=0}^{\lrf{(n-1)/2}}\gamma_{n,i}(x^2+px)^i(1+px)^i(1+2px+x^2)^{n-1-2i}.$$
Note that $(x^2+px)(1+px)=x(p(1+x)^2+(p-1)^2x),~1+2px+x^2=(1+x)^2+2(p-1)x$.
Therefore, we get
$$\widetilde{E}_n(x,p)=\sum_{i=0}^{\lrf{(n-1)/2}}\gamma_{n,i}x^i[p(1+x)^2+(p-1)^2x]^i[(1+x)^2+2(p-1)x]^{n-1-2i},$$
where $\gamma_{n,i}\geqslant 0$ for all $0\leqslant i\leqslant \lrf{(n-1)/2}$ and $p\geqslant 1$.
By using Lemma~\ref{lemma01}, we find that for $0\leqslant i\leqslant \lrf{(n-1)/2}$, the polynomials
$\gamma_{n,i}x^i[p(1+x)^2+(p-1)^2x]^i[(1+x)^2+2(p-1)x]^{n-1-2i}$ are all $\gamma$-positive with the same center of symmetry, and so $\widetilde{E}_n(x,p)$
is $\gamma$-positive, which yields the desired result. This completes the proof.
\end{proof}
\section{Excedances, fixed points and cycles of colored permutations}\label{section005}
%%%%%%%%%%%%%%%%%%%%%%%%%%%%%%%%%%%%%%%%%%%
%%%%%%%%%%%%%%%%%%%%%%%%%%%%%%%%%%%%%%%%%%%%%%%%%%%%%%%%%%%%%%%%%%%%%%%%%%%%%%%%%
%%%%%%%%%%%%%%%%%%%%%%%%%%%%%%%%%%%%%%%%%%%%%%%%%%%%%%%%%%%%%%%%%%%%%%%%%%%%%%%%%%
%%%%%%%%%%%%%%%%%%%%%%%%%%%%%%%%%%%%%%%%%%%%%%%%%%%%%%%%%%%%%%%%%%%%%%%%%%%%%%%%%%
%%%%%%%%%%%%%%%%%%%%%%%%%%%%%%%%%%%%%%%%%%%%%%%%%%
%%%%%%%%%%%%%%%%%%%%%%%%%%%%%%%%%%%%%%%%%%%%%%%%%
\subsection{Basic definitions and notation}
%%%%%%%%%%%%%%%%%%%%%%%%%%%%%%%%%%%%%%%%%%%
%%%%%%%%%%%%%%%%%%%%%%%%%%%%%%%%%%%%%%%%%%%%%%%%
\hspace*{\parindent}

Let $r$ be a fixed positive integer.
An {\it $r$-colored permutation} can be written as $\pi^c$, where $\pi=\pi_1\pi_2\cdots\pi_n\in \msn$ and
$c=(c_1,c_2,\ldots,c_n)\in [0,r-1]^n$, i.e., $c_i$ is a nonnegative integer lies in the interval $[0,r-1]$ for any $i\in[n]$. As usual, $\pi^c$ can be denoted as $\pi_1^{c_1}\pi_2^{c_2}\cdots\pi_n^{c_n}$, where $c_i$ can
be thought of as the color assigned to $\pi_i$.
Denote by $\z_r\wr\msn$ the set of all $r$-colored permutations of order $n$.
The wreath product $\z_r \wr \msn$ could be considered as the colored permutation group $G_{r,n}$ consists of all permutations of the alphabet $\Sigma$ of $rn$ letters:
$$\Sigma=\{1,2,\ldots,n,\overline{1},\ldots,\overline{n},\ldots,1^{[r-1]},\ldots,n^{[r-1]}\}$$
satisfying $\pi(\overline{i})=\overline{\pi(i)}$. In particular, $\z_1 \wr \msn=\msn$ and $\z_2 \wr \msn=\mbn$.
%%Recall that $\z_r\wr\msn$ is the set of $r$-colored permutations of order $n$.
%Let $\pi^c\in \z_r \wr \msn$.
Following Steingr\'imsson~\cite{Steingrimsson}, for $1\leqslant i\leqslant n$,
we say that an entry $\pi_i^{c_i}$ is an {\it excedance} of $\pi^c$ if $i<_f\pi_i$,
where we use the order $<_f$ of $\Sigma$:
\begin{equation}\label{order}
1<_f\overline{1}<_f\cdots<_f1^{[r-1]}<_f2<_f\overline{2}<_f\cdots<_f2^{[r-1]}<_f\cdots <_fn<_f\overline{n}<_f\cdots<_fn^{[r-1]}.
\end{equation}
%We say that an index $k\in[n]$ is a
%{\it descent} of $\pi^c$ if either $c_k>c_{k+1}$, or $c_k=c_{k+1}$ and $\pi_k>\pi_{k+1}$, where $\pi(n+1):=n+1$
%and $c_{n+1}=0$. Let $\des(\pi^c)$ be the number of descents of $\pi^c$.

Let $\exc(\pi^c)$ be the number of excedances of $\pi^c$.
A {\it fixed point} of $\pi^c\in \z_r \wr \msn$ is an entry $\pi_k^{c_k}$ such that $\pi_k=k$ and $c_k=0$.
An element $\pi^c\in \z_r \wr \msn$ is called a {\it derangement} if it has no fixed points.
Let $\D_{n,r}$ be the set of derangements in $\z_r \wr \msn$.
The {\it$q$-colored derangement polynomials} are defined by
$$d_{n,r}(x,q)=\sum_{\pi^c\in \D_{n,r}}x^{\exc(\pi^c)}q^{\cyc(\pi)}.$$
Let $d_{n,r}(x)=d_{n,r}(x,1)$ be the colored derangement polynomials.
According to~\cite[Theorem 5]{Chow10},
\begin{equation*}\label{dnrx-EGF}
\sum_{n=0}^\infty d_{n,r}(x)\frac{z^n}{n!}=\frac{(1-x)e^{(r-1)xz}}{e^{rxz}-xe^{rz}}.
\end{equation*}
There has been much work on the polynomials $d_{n,r}(x)$, see~\cite{Chow10,Zeng16} for instance.
By using the theory of Rees products of posets,
Athanasiadis~\cite[Theorem 1.3]{Athanasiadis14} obtained the following result.
\begin{theorem}[{\cite[Theorem~1.3]{Athanasiadis14}}]\label{dnr-def}
We have
\begin{equation}\label{dnrx-gamma}
d_{n,r}(x)=\sum_{i=0}^{\lrf{n/2}}\beta^+_{n,r,i}x^i(1+x)^{n-2i}+\sum_{i=0}^{\lrf{(n+1)/2}}\beta^-_{n,r,i}x^i(1+x)^{n+1-2i},
\end{equation}
where $\beta^+_{n,r,i}$ is the number of colored permutations of $\z_r \wr \msn$ for which $\Asc(\pi^c)\in [2,n]$
has exactly $i$ elements, no two consecutive, and contains $n$, and
$\beta^-_{n,r,i}$ is the number of colored permutations of $\z_r \wr \msn$ for which $\Asc(\pi^c)\in [2,n-1]$
has exactly $i-1$ elements, no two consecutive.
\end{theorem}

%%%%%%%%%%%%%%%%%%%%%%%%%%%%%%%%%%%%%%%%%%%%%%%%%%%%%%%%%%%%%%%%%%%%%%%%%%%%%%%%%%
%%%%%%%%%%%%%%%%%%%%%%%%%%%%%%%%%%%%%%%%%%%%%%%%%%%%%%%%%%%%%%%%%%%%%%%%%%%%%%%%%%
%%%%%%%%%%%%%%%%%%%%%%%%%%%%%%%%%%%%%%%%%%%%%%%%%%
%%%%%%%%%%%%%%%%%%%%%%%%%%%%%%%%%%%%%%%%%%%%%%%%%
\subsection{An equivalent result of Theorem~\ref{thm001}}
%%%%%%%%%%%%%%%%%%%%%%%%%%%%%%%%%%%%%%%%%%%
%%%%%%%%%%%%%%%%%%%%%%%%%%%%%%%%%%%%%%%%%%%%%%%%
\hspace*{\parindent}

In~\cite[Proposition~4]{Chow10}, Chow and Toufik found that
\begin{equation*}\label{dnr1}
d_{n,r}(x)=\sum_{\pi\in\msn}x^{\wexc(\pi)}(r-1)^{\fix(\pi)}r^{n-\fix(\pi)},
\end{equation*}
where $\wexc(\pi)=\#\{i\in[n]: \pi(i)\geqslant i\}$ is the number of {\it weak excedances} of $\pi$.
Note that $\wexc(\pi)=\exc(\pi)+\fix(\pi)$.
Along the same lines as in the proof of~\cite[Proposition~4]{Chow10}, one can easily derive that
\begin{equation*}\label{dnrq}
d_{n,r}(x,q)=\sum_{\pi\in\msn}x^{\wexc(\pi)}(r-1)^{\fix(\pi)}r^{n-\fix(\pi)}q^{\cyc(\pi)}.
\end{equation*}

Define $$a_n(x,p,q)=\sum_{\pi\in\msn}x^{\wexc(\pi)}p^{\fix(\pi)}q^{\cyc(\pi)}.$$
Then $a_n(x,p,q)=A_n(x,xp,q)$.
Let $\pi^{-1}$ be the inverse of $\pi$. The bijection $\pi\rightarrow \pi^{-1}$ on $\msn$ shows that
$(\exc,\fix,\cyc)$ is equidistributed with $(\drop,\fix,\cyc)$. Thus $(\wexc,\fix,\cyc)$ is equidistributed with $(n-\exc,\fix,\cyc)$.
Therefore, we get $$a_n(x,p,q)=x^nA_n\left(\frac{1}{x},p,q\right).$$
It should be noted that if $p>0$ and $q>0$, then $\deg a_n(x,p,q)=n$.
Moreover, $a_n(0,p,q)=0$ and the coefficient of the highest degree term of $a_n(x,p,q)$ is $p^nq^n$, and so $a_n(0,p,q)< p^nq^n$ when $p>0$ and $q>0$.
Therefore, an equivalent result of Theorem~\ref{thm001} is given as follows.
\begin{theorem}\label{cor-anxp}
Let $p\in[0,1]$ and $q\in [0,1]$ be two given real numbers. The polynomials $a_n(x,1,q)$ are bi-$\gamma$-positive and the polynomials $a_n(x,p,q)$ are alternatingly increasing for $n\geqslant 1$.
\end{theorem}
%Shin and Zeng~\cite[Theorem 3]{Zeng16} obtained the following expansion:
%\begin{equation*}
%d_{n,r}(x,1)=\sum_{1\leqslant i+2j\leqslant n}\gamma_{n,i,j}x^{i+j}(1+x)^{n-i-2j}(r-1)^ir^{n-i}.
%\end{equation*}

Note that $$d_{n,r}(x,q)=r^na_n\left(x,\frac{r-1}{r},q\right).$$
As a special case of Theorem~\ref{cor-anxp}, we get the following result.
\begin{corollary}
Let $q\in[0,1]$ be a given real number and let $r$ be a fixed positive integer. Then
the polynomials $d_{n,r}(x,q)$ are alternatingly increasing for $n\geqslant 1$.
\end{corollary}

%%%%%%%%%%%%%%%%%%%%%%%
\subsection{The bi-$\gamma$-positivity of colored Eulerian polynomials and an application}
\hspace*{\parindent}
%%%%%%%%%%%%%%%%%%%%%%%
%%%%%%%%%%%%%%%%%%%%%%%

Following Steingr\'imsson~\cite{Steingrimsson},
the {\it $r$-colored Eulerian polynomial} is defined by
$$A_{n,r}(x)=\sum_{\pi^c\in\z_r \wr \msn}x^{\exc(\pi^c)},$$
which satisfy the recurrence relation
\begin{equation}\label{Anrx-recu}
A_{n,r}(x)=(1+(rn-1)x)A_{n-1,r}(x)+rx(1-x)\frac{\mathrm{d}}{\mathrm{d}x}A_{n,r}(x),
\end{equation}
with $A_{0,r}(x)=1,A_{1,r}(x)=1+(r-1)x,~A_{2,r}(x)=1+(r^2+2r-2)x+(r-1)^2x^2$.
%A_{3,r}(x)&=1+(r^3+3r^2+3r-3)x+(3-6r+4r^3)x^2+(r-1)^3x^3.
Let $A_{n,r}(x)=\sum_{k=0}^nA_r(n,k)x^k$. It follows from~\eqref{Anrx-recu} that
the numbers $A_r(n,k)$ satisfy the recurrence
\begin{equation}\label{Arnk-recu}
A_r(n,k)=(rk+1)A_r(n-1,k)+(r(n-k)+r-1)A_r(n-1,k-1),
\end{equation}
with $A_r(0,k)=\delta_{0,k}$.
%Following~\cite[Lemma~3.5]{Savage15}, we have $$A_{n,r}(x)=E_n^{(r,2r,\ldots,nr)}(x).$$
According to~\cite[Theorem~20]{Steingrimsson}, we have
\begin{equation}\label{Anr-EGF}
\sum_{n=0}^\infty A_{n,r}(x)\frac{z^n}{n!}=\frac{(1-x)\mathrm{e}^{z(1-x)}}{1-x\mathrm{e}^{rz(1-x)}}.
\end{equation}
When $r=1$ and $r=2$, the polynomial $A_{n,r}(x)$ reduces to the types $A$ and $B$ Eulerian polynomials $A_n(x)$ and $B_n(x)$, respectively.
%The $\gamma$-positivities of $A_{n,1}(x)$ and $A_{n,2}(x)$, along with several $q$-analogues, were discovered repeatedly, see, e.g.,~\cite{Athanasiadis17,Lin15}.

By using the principle of inclusion-exclusion, one has $A_{n,r}(x)=\sum_{k=0}^n\binom{n}{k}d_{k,r}(x)$.
Combining this with~\eqref{dnrx-gamma},
Athanasiadis~\cite[Eq.~(21)]{Athanasiadis20} recently obtained the following expansion:
\begin{equation}\label{Anrx01}
A_{n,r}(x)=A_{n,r}^+(x)+A_{n,r}^-(x),
\end{equation}
where $$A_{n,r}^+(x)=\sum_{k=0}^n\binom{n}{k}d_{k,r}^+(x),~A_{n,r}^-(x)=\sum_{k=0}^n\binom{n}{k}d_{k,r}^-(x),$$
$$d_{n,r}^+(x)=\sum_{i=0}^{\lrf{n/2}}\beta^+_{n,r,i}x^i(1+x)^{n-2i},~d_{n,r}^-(x)=\sum_{i=1}^{\lrf{(n+1)/2}}\beta^-_{n,r,i}x^i(1+x)^{n+1-2i}.$$
Comparing~\eqref{Anxpq-EGF} with~\eqref{Anr-EGF}, one can easily find that
\begin{equation}\label{AnrxAn}
A_{n,r}(x)=r^nA_n\left(x,\frac{1+(r-1)x}{r},1\right).
\end{equation}
In particular, we have $$B_n(x)=A_{n,2}(x)=\sum_{\pi\in\msn}x^{\exc(\pi)}(1+x)^{\fix(\pi)}2^{n-\fix(\pi)}.$$
Motivated by~\eqref{Anrx01} and~\eqref{AnrxAn}, in the following we shall first present a direct proof of the bi-$\gamma$-positivity of $A_{n,r}(x)$,
and then we study some multivariate colored Eulerian polynomials.
%%%%%%%%%%%%%%%%%%%%%%%%%%%%%%%%%%%%%%%%%%%%%
%%%%%%%%%%%%%%%%%%%%%%%%%%%%%%%%%%%%%%%%%%%%%%%%%%%%%%%%%%%%%%%%%%%%%%%%%%%%%%%%%%%%
%%%%%%%%%%%%%%%%%%%%%%%%%%%%%%%%%%%%%%%%%%%%%%%%%%%%%%%%%%%%%%%%%%%%%%%%%%%%%%%%%%%%

%The following result gives a grammatical interpretation of the numbers $A_r(n,k)$.
\begin{lemma}\label{lemma001}
If
$G_8=\{u\rightarrow uv^r,v\rightarrow u^rv\}$,
then we have
\begin{equation}\label{grammar01G1}
D_{G_8}^n(u^{r-1}v)=u^{r-1}v\sum_{k=0}^nA_r(n,k)u^{(n-k)r}v^{kr}.
\end{equation}
\end{lemma}
\begin{proof}
Note that $D_{G_8}^0(u^{r-1}v)=u^{r-1}v$ and $D_{G_8}(u^{r-1}v)=u^{r-1}v(u^r+(r-1)v^r)$.
Assume that the result holds for $n=m$, where $m\geqslant 1$. Then
\begin{align*}
&D_{G_8}^{m+1}(u^{r-1}v)\\
&=D_{G_8}\left(u^{r-1}v\sum_{k=0}^mA_r(m,k)u^{(m-k)r}v^{kr}\right)\\
&=u^{r-1}v\sum_{k}A_r(m,k)((mr-kr+r-1)u^{(m-k)r}v^{(k+1)r}+(kr+1)u^{(m-k+1)r}v^{kr}).
\end{align*}
So we get that
$A_r(m+1,k)=(rk+1)A_r(m,k)+(r(m+1-k)+r-1)A_r(m,k-1)$. It follows from~\eqref{Arnk-recu} that the result holds for $n=m+1$. This completes the proof.
\end{proof}

Recall that $A_{n,1}(x)=A_n(x)$ and $A_{n,2}(x)=B_n(x)$, which are both $\gamma$-positive polynomials.
We can now present the sixth main result of this paper.
\begin{theorem}\label{lemma-Eulerian}
For each $r\geqslant 2$,
the polynomial $A_{n,r}(x)$ are bi-$\gamma$-positive. More precisely,
we have
\begin{equation}\label{mainformula01}
A_{n,r}(x)=\sum_{k=0}^{\lrf{n/2}}\alpha^+_{n,k;r}x^k(1+x)^{n-2k}+x\sum_{k=0}^{\lrf{(n-1)/2}}\alpha^-_{n,k;r}x^k(1+x)^{n-1-2k},
\end{equation}
where the numbers $\alpha^+_{n,k;r}$ and $\alpha^-_{n,k;r}$ satisfy the recurrence system
\begin{equation*}\label{recu-ankr02}
\left\{
  \begin{array}{ll}
   \alpha^+_{n+1,k;r}=(1+rk)\alpha^+_{n,k;r}+2r(n-2k+2)\alpha^+_{n,k-1;r}+2\alpha^-_{n,k-1;r}, & \\
    \alpha^-_{n+1,k;r}=(r-2)\alpha^+_{n,k;r}+(r-1+rk)\alpha^-_{n,k;r}+2r(n-2k+1)\alpha^-_{n,k-1;r},
  \end{array}
\right.
\end{equation*}
with the initial conditions $\alpha^+_{1,0;r}=1,~\alpha^-_{1,0;r}=r-2$, $\alpha^+_{1,k;r}=\alpha^-_{1,k;r}=0$ for $k\neq 0$.
\end{theorem}
\begin{proof}
Consider a change of the grammar given in Lemma~\ref{lemma001}.
Note that
\begin{align*}
D_{G_8}(u^rv^r)&=ru^rv^r(u^r+v^r),~D_{G_8}(u^r+v^r)=2ru^rv^r,\\
D_{G_8}(u^{r-1}v)&=(r-2)u^{r-1}v^{r+1}+u^{r-1}v(u^r+v^r),\\
D_{G_8}(u^{r-1}v^{r+1})
&=(r-1)u^{r-1}v^{r+1}(u^r+v^r)+2u^{r-1}v(u^rv^r).
\end{align*}
Setting $a=u^rv^r,~b=u^r+v^r,~c=u^{r-1}v^{r+1}$ and $I=u^{r-1}v$, we obtain
$D_{G_8}(a)=rab,~D_{G_8}(b)=2ra,~D_{G_8}(c)=(r-1)bc+2Ia,~D_{G_8}(I)=(r-2)c+Ib$.
Clearly, $c=Iv^r$.
Consider the grammar
$$G_9=\{I\rightarrow Ib+(r-2)c,~a\rightarrow rab,~b\rightarrow 2ra,~c\rightarrow (r-1)bc+2Ia\}.$$
Note that $D_{G_9}(I)=Ib+(r-2)c,~D_{G_9}^2(I)=(4(r-1)a+b^2)I+r(r-2)bc$.
%$$D_{G_1}^3(I)=[(6r^2-4)ab+b^3]I+[2(r-2)(r^2+2r-2)a+(r-2)(r^2-r+1)b^2]c.$$
Then by induction, it is routine to check that there exist nonnegative integers
such that
\begin{equation}\label{grammar02G2}
D_{G_9}^n(I)=\sum_{k=0}^{\lrf{n/2}}\alpha^+_{n,k;r}a^kb^{n-2k}I+\sum_{k=0}^{\lrf{(n-1)/2}}\alpha^-_{n,k;r}a^kb^{n-1-2k}c.
\end{equation}
We proceed to the inductive step.
Note that
\begin{align*}
D_{G_9}^{n+1}(I)&=\sum_{k}\alpha^+_{n,k;r}b^{n-1-2k}((1+rk)a^kb^{2}I+2r(n-2k)a^{k+1}I+(r-2)a^kbc)+\\
&\sum_{k}\alpha^-_{n,k;r}b^{n-2-2k}((r-1+rk)a^kb^{2}c+2r(n-1-2k)a^{k+1}c+2a^{k+1}bI).
\end{align*}
Taking coefficients of $a^kb^{n+1-2k}I$ and $a^kb^{n-2k}c$ on both sides and simplifying yields the desired
recurrence system.
Setting $u^r=1$ and $v^r=x$, we have $a=x,~b=1+x$ and $c=Ix$.
Comparing~\eqref{grammar01G1} and~\eqref{grammar02G2}, we immediately get~\eqref{mainformula01}.
\end{proof}

Define
$$A^+_{n,r}(x)=\sum_{k=0}^{\lrf{n/2}}\alpha^+_{n,k;r}x^k(1+x)^{n-2k},~\widetilde{A}^-_{n,r}(x)=\sum_{k=0}^{\lrf{(n-1)/2}}\alpha^-_{n,k;r}x^k(1+x)^{n-1-2k}.$$
Comparing this with~\eqref{Anrx01}, we obtain $A_{n,r}(x)=A^+_{n,r}(x)+x\widetilde{A}^-_{n,r}(x)$ and $A^-_{n,r}(x)=x\widetilde{A}^-_{n,r}(x)$.
It is routine to verify the following corollary.
\begin{corollary}\label{cor-Anrx}
For $n\geqslant 1$, the polynomials $A^+_{n,r}(x)$ and $\widetilde{A}^-_{n,r}(x)$ satisfy the recurrence system
\begin{align*}
A^+_{n+1,r}(x)&=(1+x+rnx)A^+_{n,r}(x)+rx(1-x)\frac{\mathrm{d}}{\mathrm{d}x}A^+_{n,r}(x)+2x\widetilde{A}^-_{n,r}(x),\\
\widetilde{A}^-_{n+1,r}(x)&=(r-1+(rn-1)x)\widetilde{A}^-_{n,r}(x)+rx(1-x)\frac{\mathrm{d}}{\mathrm{d}x}\widetilde{A}^-_{n,r}(x)+(r-2)A^+_{n,r}(x),
\end{align*}
with the initial conditions $A^+_{0,r}(x)=1$ and $\widetilde{A}^-_{0,r}(x)=0$.
\end{corollary}

In the rest of this subsection, we present an application of Theorem~\ref{lemma-Eulerian}.
For $\sigma\in\mbn$, let $\des_B(\sigma)=\#\{i\in\{0,1,\ldots,n-1\}\mid\sigma(i)>\sigma({i+1})\}$, where $\sigma(0)=0$.
Following Brenti~\cite[Corollary~3.16]{Brenti94},
the {\it type $B$ $q$-Eulerian polynomials} can be defined as follows:
$$B_n(x,q)=\sum_{\pi\in \mbn}x^{\des_B(\pi)}q^{\negg(\sigma)}=\sum_{\pi\in \mbn}x^{\exc(\sigma)+\fix(\sigma)}q^{n-\negg(\pi)}.$$
The polynomials $B_n(x,q)$ satisfy the recurrence relation
\begin{equation*}
B_n(x,q)=(1+(1+q)nx-x)B_{n-1}(x,q)+(1+q)(x-x^2)\frac{\mathrm{\partial}}{\mathrm{\partial} x}B_{n-1}(x,q).
\end{equation*}
with $B_0(x,q)=1,B_1(x,q)=1+qx$ and $B_2(x,q)=1+(1+4q+q^2)x+q^2x^2$, and the exponential generating function of $B_n(x,q)$ is given as follows (see~\cite[Theorem~3.4]{Brenti94}):
\begin{equation}\label{BnxEGF}
\sum_{n=0}^{\infty}B_n(x,q)\frac{z^n}{n!}=\frac{(1-x)\mathrm{e}^{z(1-x)}}{1-x\mathrm{e}^{z(1-x)(1+q)}}.
\end{equation}

Comparing~\eqref{Anr-EGF} with~\eqref{BnxEGF}, one has $A_{n,q+1}(x)=B_n(x,q)$.
Let $B_n(x,q)=\sum_{k=0}^nB_{n,k}(q)x^k$ and let $$q^nB_n\left(x,{1}/{q}\right)=\sum_{k=0}^n\widetilde{B}_{n,k}(q)x^k,$$ where $q>0$. It follows
from~\cite[Corollary~3.16]{Brenti94} that
$B_{n,k}(q)=\widetilde{B}_{n,n-k}(q)$.
By using Theorem~\ref{lemma-Eulerian}, we get the following.
\begin{corollary}\label{mthm03}
If $q\geqslant 1$, then the polynomial $B_n(x,q)$ is alternatingly increasing. If $0\leqslant q\leqslant 1$,
then $B_n(x,q)$ is spiral, i.e.,
$$B_{n,n}(q)\leqslant B_{n,0}(q)\leqslant B_{n,n-1}(q)\leqslant B_{n,1}(q)\leqslant \cdots\leqslant B_{n,\lrf{n/2}}(q).$$
\end{corollary}
%%%%%%%%%%%%%%%%%%%%%%%
\subsection{The first kind of multivariate colored Eulerian polynomials}
\hspace*{\parindent}

%Let $\pi^c=\pi_1^{c_1}\pi_2^{c_2}\cdots\pi_n^{c_n}\in \z_r \wr \msn$.
%We define
%\begin{align*}
%\bexc(\pi^c)&=\#\{i\in[n]:i<_f \pi_i~{\text{and}}~c_i=0\},~\aexc(\pi^c)=\#\{i\in[n]:\pi_i{<_c}i\},\\
%\fix(\pi^c)&=\#\{i\in[n]:\pi_i=i~{\text{and}}~c_i=0\},~\single(\pi^c)=\#\{i\in[n]:\pi_i=i~{\text{and}}~c_i>0\},\\
%\csum(\pi^c)&=\sum_{i=1}^nc_i,~\fexc(\pi^c)=r\cdot\exc_A(\pi^c)+\csum(\pi^c),
%\end{align*}
%where the comparison is with respect to the order $<_c$ of $\Sigma$:
Let $\pi^c\in \z_r \wr \msn$.
Recall that $\exc(\pi^c)=\#\{i\in[n]:i<_f\pi_i\}$. We define
\begin{align*}
\aexc(\pi^c)&=\#\{i\in[n]:\pi_i<_f i\},~\fix(\pi^c)=\#\{i\in[n]:\pi_i=i~{\text{and}}~c_i=0\}.
\end{align*}
%We define
%\begin{align*}
%\bexc(\pi^c)&=\#\{i\in[n]:i<{\pi_i}\},~\aexc(\pi^c)=\#\{i\in[n]:\pi_i{<_f}i\},\\
%\fix(\pi^c)&=\#\{i\in[n]:\pi_i=i~{\text{and}}~c_i=0\},~\single(\pi^c)=\#\{i\in[n]:\pi_i=i~{\text{and}}~c_i>0\}.
%\end{align*}
%%where the comparison is with respect to the order $<_c$ of $\Sigma$:
%Let $\aexc(\pi^c)=\#\{i\in[n]:\pi_i<_f i\}$ be the number of {\it anti-excedances} of $\pi^c$.
Consider the following {\it multivariate colored Eulerian polynomials}:
$$A_{n,r}(x,y,p,q)=\sum_{\pi^c\in \z_r \wr \msn}x^{\exc(\pi^c)}y^{\aexc(\pi^c)}p^{\fix(\pi^c)}q^{\cyc(\pi^c)}.$$
Clearly, $A_{n,1}(x,1,p,q)=A_n(x,p,q),A_{n,r}(x,1,1,1)=A_{n,r}(x)$ and $A_{n,r}(x,1,0,q)=d_{n,r}(x,q)$.

\begin{lemma}\label{derangment-grammar001}
If $G_{10}=\{I\rightarrow qI\left((r-1)x+p\right),x\rightarrow rxy,y\rightarrow rxy,p\rightarrow rxy\}$,
then
\begin{equation}\label{derangment-grammar022}
D_{G_{10}}^n(I)=I\sum_{\pi^c\in \z_r \wr \msn}x^{\exc(\pi^c)}y^{\aexc(\pi^c)}p^{\fix(\pi^c)}q^{\cyc(\pi^c)}.
\end{equation}
\end{lemma}
\begin{proof}
We now introduce a grammatical labeling of $\pi^c\in \z_r \wr \msn$ as follows:
\begin{itemize}
  \item [\rm ($L_1$)]If $i<_f\pi_i$, then we label $\pi_i^{c_i}$ by a subscript label $x$;
 \item [\rm ($L_2$)]If $\pi_i<_f i$, then we label $\pi_i^{c_i}$ by a subscript label $y$;
\item [\rm ($L_3$)]If $\pi_i=i$ and $c_i=0$, then we label $i$ by a subscript label $p$;
\item [\rm ($L_4$)]Put a subscript label $I$ right after $\pi^c$, and put a superscript label $q$ right after each cycle.
\end{itemize}
Note that the weight of $\pi^c$ is given by $w(\pi^c)=Ix^{\exc(\pi^c)}y^{\aexc(\pi^c)}p^{\fix(\pi^c)}q^{\cyc(\pi^c)}$.
For $n=1$, we have $$\z_r \wr \ms_1=\{(1_p)^q_I,(\overline{1}_x)^q_I,(\overline{\overline{1}}_x)^q_I,\ldots,({1^{[r-1]}}_x)^q_I\}.$$
Note that $D_{G_{10}}(I)=qI((r-1)x+p)$.
Then the sum of weights of the elements in $\z_r \wr \ms_1$ is given by $D_{G_{10}}(I)$.
Hence the result holds for $n=1$.
We proceed by induction on $n$.
Suppose we get all labeled permutations in $\pi^c\in \z_r \wr \ms_{n-1}$, where $n\geqslant 2$. Let
$\widehat{{\pi^c}}$ be obtained from $\pi^c\in \z_r \wr \ms_{n-1}$ by inserting $n^{c_j}$, where $0\leqslant c_j\leqslant r-1$ is a nonnegative integer.
When the inserted $n^{c_j}$ forms a new cycle, the insertion corresponds to the substitution rule $I\rightarrow qI((r-1)x+p)$ since we have $r$ choices for $c_j$.
For the other cases, the changes of labeling are illustrated as follows:
$$\cdots(\cdots{\pi_{i}^{c_{i}}}_x\pi_{i+1}^{c_{i+1}}\cdots)\cdots\mapsto \cdots(\cdots{\pi_{i}^{c_{i}}}_x{n^{c_j}}_y\pi_{i+1}^{c_{i+1}}\cdots)\cdots;$$
$$\cdots(\cdots{\pi_{i}^{c_{i}}}_y\pi_{i+1}^{c_{i+1}}\cdots)\cdots\mapsto \cdots(\cdots{\pi_{i}^{c_{i}}}_x{n^{c_j}}_y\pi_{i+1}^{c_{i+1}}\cdots)\cdots;$$
$$\cdots({i}~_p)\cdots\mapsto \cdots({i}~_x{n^{c_j}}_y)\cdots.$$
In each case, the insertion of $n^{c_j}$ corresponds to one substitution rule in $G_{10}$. By induction,
it is routine to check that the action of $D_{G_{10}}$
on the set of weights of colored permutations in $\z_r \wr \ms_{n-1}$ gives the set of weights of colored permutations in $\z_r \wr \ms_{n}$.
This completes the proof.
\end{proof}

As a generalization of~\eqref{AnrxAn}, we now conclude the following result.
\begin{theorem}\label{thm32}
One has
\begin{equation*}
A_{n,r}(x,y,p,q)=\sum_{\pi\in\msn}(rx)^{\exc(\pi)}(ry)^{\drop(\pi)}((r-1)x+p)^{\fix(\pi)}q^{\cyc(\pi)}.
\end{equation*}
Equivalently,
\begin{equation}\label{Anrxypq}
A_{n,r}(x,y,p,q)=(ry)^nA_n\left(\frac{x}{y},\frac{(r-1)x+p}{ry},q\right).
\end{equation}
\end{theorem}
\begin{proof}
Let $G_{10}$ be the grammar given in Lemma~\ref{derangment-grammar001}.
Setting $a=(r-1)x+p,b=rx$ and $c=ry$, we get
$D_{G_{10}}(I)=qIa,~D_{G_{10}}(a)=bc,~D_{G_{10}}(b)=bc,~D_{G_{10}}(c)=bc$.
Let $G_{11}=\{I\rightarrow qIa, a\rightarrow bc, b\rightarrow bc, c\rightarrow bc\}$.
It follows from Lemma~\ref{lemma001exc} that
\begin{equation}\label{DG08}
D_{G_{11}}^n(I)=I\sum_{\pi\in\msn}b^{\exc(\pi)}c^{\drop(\pi)}a^{\fix(\pi)}q^{\cyc(\pi)}.
\end{equation}
Then upon taking $a=(r-1)x+p,b=rx$ and $c=ry$ in~\eqref{DG08}, we immediately get the desired result. This completes the proof.
\end{proof}

From~\eqref{Anrxypq}, we see that $A_{n,r}(x,1,x,q)=r^nA_n\left(x,x,q\right)$.
Combining this result with Theorem~\ref{cor-anxp}, we get the following corollary.
\begin{corollary}
Let $q\in[0,1]$ be a given real number. Then the polynomials $A_{n,r}(x,1,x,q)$ are bi-$\gamma$-positive for $n\geqslant 1$.
\end{corollary}

In the rest part of this subsection, we shall give a further generalization of~\eqref{AnrxAn}.
As usual, set $[0]_p=0$.
For any positive integer $n$, let $$[n]_p=1+p+\cdots+p^{n-1}.$$ Let $\pi^c\in \z_r \wr \msn$.
A {\it singleton} of $\pi^c$ is an entry $\pi_i^{c_i}$ such that $\pi_i=i$ and $c_i>0$.
We define
\begin{align*}
\exc_B(\pi^c)&=\#\{i\in[n]:i<_f\pi_i~\text{and}~\pi_i\neq i\},\\
\single(\pi^c)&=\#\{i\in[n]:\pi_i=i~{\text{and}}~c_i>0\}.
\end{align*}
It is clear that $\exc(\pi^c)=\exc_B(\pi^c)+\single(\pi^c)$.
Let $\csum(\pi^c)=\sum_{i=1}^nc_i$.
Consider the following {\it multivariate colored Eulerian polynomials}:
$$A_{n,r}(x,y,s,t,p,q)=\sum_{\pi^c\in \z_r \wr \msn}x^{\exc_B(\pi^c)}y^{\aexc(\pi^c)}s^{\single(\pi^c)}t^{\fix(\pi^c)}p^{\csum(\pi^c)}q^{\cyc(\pi^c)}.$$

As a refinement of the grammatical labeling given in the proof of Lemma~\ref{derangment-grammar001}, we give another grammatical labeling of $\pi^c\in \z_r \wr \msn$ as follows:
\begin{itemize}
  \item [\rm ($L_1$)]If $i<_f\pi_i$ and $\pi_i\neq i$, then we label $\pi_i^{c_i}$ by a subscript label $x$;
 \item [\rm ($L_2$)]If $\pi_i<_f i$, then we label $\pi_i^{c_i}$ by a subscript label $y$;
\item [\rm ($L_3$)]If $\pi_i=i$ and $c_i=0$, then we put a subscript label $t$ just before $i$;
\item [\rm ($L_4$)]If $\pi_i=i$ and $c_i>0$, then we put a subscript label $s$ just before $\pi_i^{c_i}$;
 \item [\rm ($L_5$)] Put a subscript label $p^{c_i}$ right after each element $\pi_i^{c_i}$ of $\pi^c$;
\item [\rm ($L_6$)]Put a subscript label $I$ right after $\pi^c$, and put a superscript label $q$ right after each cycle.
\end{itemize}
Then the weight of $\pi^c$ is given by $$w(\pi^c)=Ix^{\exc_B(\pi^c)}y^{\aexc(\pi^c)}s^{\single(\pi^c)}t^{\fix(\pi^c)}p^{\csum(\pi^c)}q^{\cyc(\pi^c)}.$$
For $n=1$, we have $\z_r \wr \ms_1=\{(_t1_{p^0})^q_I,(_s\overline{1}_{p})^q_I,(_s\overline{\overline{1}}_{p^2})^q_I,\ldots,(_s{1^{[r-1]}}_{p^{r-1}})^q_I\}$.
By induction, it is routine to check the following result and
we omit the proof of it for simplicity.
\begin{lemma}\label{grammar0034}
If $G_{12}=\{I\rightarrow qI\left(t+sp[r-1]_p\right),x\rightarrow [r]_pxy,y\rightarrow [r]_pxy,t\rightarrow [r]_pxy,s\rightarrow [r]_pxy\}$,
then
\begin{equation}\label{grammar034}
D_{G_{12}}^n(I)=I\sum_{\pi^c\in \z_r \wr \msn}x^{\exc_B(\pi^c)}y^{\aexc(\pi^c)}s^{\single(\pi^c)}t^{\fix(\pi^c)}p^{\csum(\pi^c)}q^{\cyc(\pi^c)}.
\end{equation}
\end{lemma}

%\begin{lemma}\label{grammar035}
%\begin{equation*}
%D_{G_10}^n(I)=I\sum_{\pi^c\in \z_r \wr \msn}x^{\exc_B(\pi^c)}y^{\aexc(\pi^c)}s^{\single(\pi^c)}t^{\csum(\pi^c)}p^{\fix(\pi^c)}q^{\cyc(\pi^c)}.
%\end{equation*}
%\end{lemma}
We can now present the seventh main result of this paper.
\begin{theorem}
We have
\begin{equation*}
A_{n,r}(x,y,s,t,p,q)=[r]_p^ny^nA_n\left(\frac{x}{y},\frac{t+sp[r-1]_p}{[r]_py},q\right).
\end{equation*}
In particular,
\begin{equation*}\label{Anr1x1tpq}
A_{n,r}(x,1,x,t,p,q)=\sum_{\pi^c\in \z_r \wr \msn}x^{\exc(\pi^c)}t^{\fix(\pi^c)}p^{\csum(\pi^c)}q^{\cyc(\pi^c)}=[r]_p^nA_n\left(x,\frac{t+xp[r-1]_p}{[r]_p},q\right).
\end{equation*}
\end{theorem}
\begin{proof}
Let $G_{12}$ be the grammar given in Lemma~\ref{grammar0034}.
Setting $a_1=t+sp[r-1]_p,~a_2=[r]_px,~a_3=[r]_py$, we get
$$D_{G_{12}}(a_1)=a_2a_3,~D_{G_{12}}(a_2)=a_2a_3,~D_{G_{12}}(a_3)=a_2a_3.$$
Let $G_{13}=\{I\rightarrow qIa_1, a_1\rightarrow a_2a_3, a_2\rightarrow a_2a_3, a_3\rightarrow a_2a_3\}$.
By Lemma~\ref{lemma001exc}, we get
$$D_{G_{13}}^n(I)=I\sum_{\pi\in\msn}a_1^{\fix(\pi)}a_2^{\exc(\pi)}a_3^{\drop(\pi)}q^{\cyc(\pi)}=Ia_3^nA_n\left(\frac{a_2}{a_3},\frac{a_1}{a_3},q\right).$$
Then upon taking $a_1=t+sp[r-1]_p,~a_2=[r]_px$ and $a_3=[r]_py$ in the above expression, we immediately get the desired result. This completes the proof.
\end{proof}
\begin{corollary}
Let $\wexc(\pi^c)=\exc(\pi^c)+\fix(\pi^c)$.
We have $$\sum_{\pi^c\in \z_r \wr \msn}x^{\wexc(\pi^c)}p^{\csum(\pi^c)}q^{\cyc(\pi^c)}=[r]_p^n\sum_{\pi\in\msn}x^{\wexc(\pi)}q^{\cyc(\pi)}.$$
\end{corollary}

\subsection{The second kind of multivariate colored Eulerian polynomials}
\hspace*{\parindent}

Let $\pi^c=\pi_1^{c_1}\pi_2^{c_2}\cdots\pi_n^{c_n}\in \z_r \wr \msn$.
Following~\cite{Athanasiadis14,Bagno04,Zeng16}, we define
\begin{align*}
\exc_A(\pi^c)&=\#\{i\in[n]:i<_c{\pi_i}~{\text{and}}~c_i=0\},~\aexc_A(\pi^c)=\#\{i\in[n]:\pi_i{<_c}i\},\\
\fix(\pi^c)&=\#\{i\in[n]:\pi_i=i~{\text{and}}~c_i=0\},~\single(\pi^c)=\#\{i\in[n]:\pi_i=i~{\text{and}}~c_i>0\},\\
\csum(\pi^c)&=\sum_{i=1}^nc_i,~\fexc(\pi^c)=r\cdot\exc_A(\pi^c)+\csum(\pi^c),
\end{align*}
where the comparison is with respect to the order $<_c$ of $\Sigma$:
$$1^{[r-1]}<_c2^{[r-1]}<_c\cdots<_cn^{[r-1]}<_c\cdots<_c\overline{1}<_c\overline{2}<_c\cdots<_c\overline{n}<_c1<_c2<_c\cdots<_cn.$$
%We say that an entry $\pi_i^{c_i}$ is an {\it anti-excedance} of $\pi^c$ if $\pi_i<i$.
%
%Let $\cyc(\pi^c)$ be the number of cycles of $\pi^c$.
%More information about the statistics over $\z_r \wr \msn$, and references can be found in~\cite{Athanasiadis14,Athanasiadis20,Zeng16}.
Consider the following polynomials
$$A_{n}^{(r)}(x,y,s,t,p,q)=\sum_{\pi^c\in \z_r \wr \msn}x^{\exc_A(\pi^c)}y^{\aexc_A(\pi^c)}s^{\single(\pi^c)}t^{\fix(\pi^c)}p^{\csum(\pi^c)}q^{\cyc(\pi^c)}.$$
For $1\leqslant i\leqslant n$, we introduce a grammatical labeling of $\pi^c$ as follows:
\begin{itemize}
 \item [\rm ($L_1$)] Put a subscript label $p^{c_i}$ right after each element $\pi_i^{c_i}$ of $\pi^c$;
  \item [\rm ($L_2$)]If $\pi_i=i$ and $c_i=0$, then put a superscript label $t$ right after $i$;
    \item [\rm ($L_3$)] If $\pi_i=i$ and $c_i>0$, then put a superscript label $s$ right after $\pi_i^{c_i}$;
 \item [\rm ($L_4$)]If $i<_c\pi_i$ and $c_i=0$, then we label $\pi_i$ by $x$;
  \item [\rm ($L_5$)]If $\pi_i<_ci$, then we label $\pi_i^{c_i}$ by $y$;
\item [\rm ($L_6$)]Put a subscript label $I$ right after $\pi^c$ and put a superscript label $q$ right after each cycle.
\end{itemize}
In particular, the grammatical labeling of elements in $\z_r \wr \ms_1$ are illustrated as follows:
$$\z_r \wr \ms_1=\{(1_{p^0}^t)_I^q,{(\overline{1}_{p}^s})_I^q,{({\overline{\overline{1}}}^s_{p^2}})_I^q,\ldots,({{1^{[r-1]}}^s_{p^{r-1}}})_I^q\}.$$
We now provide an example to illustrate the above grammatical labeling.
\begin{example}
Let $\pi^c=(1,4,\overline{5},{2})(\overline{\overline{3}})\in \z_3 \wr \ms_5$. The grammatical labeling of $\pi^c$ is given below
$$(1_{p^0}^x4_{p^0}^y\overline{5}_{p}^y{2}_{p^0}^y)^q(\overline{\overline{3}}^s_{p^2})^q_I.$$
Note that $c_i=0,1$ or $2$.
If we insert $6^{c_i}$ into $\pi^c$ as a new cycle, we get the following permutations:
$$(1_{p^0}^x4_{p^0}^y\overline{5}_{p}^y{2}_{p^0}^y)^q(\overline{\overline{3}}^s_{p^2})^q(6_{p_0}^t)^q_I,
~(1_{p^0}^x4_{p^0}^y\overline{5}_{p}^y{2}_{p^0}^y)^q(\overline{\overline{3}}^s_{p^2})^q(\overline{6}_{p}^s)^q_I,~
(1_{p^0}^x4_{p^0}^y\overline{5}_{p}^y{2}_{p^0}^y)^q(\overline{\overline{3}}^s_{p^2})^q(\overline{\overline{6}}_{p^2}^{s})^q_I.$$
If we insert $6^{c_i}$ right after the element $1$, we get the following permutations:
$$(1_{p^0}^x6_{p^0}^y4_{p^0}^y\overline{5}_{p}^y{2}_{p^0}^y)^q(\overline{\overline{3}}^s_{p^2})^q_I,
~(1_{p^0}^y\overline{6}_{p}^y4_{p^0}^y\overline{5}_{p}^y{2}_{p^0}^y)^q(\overline{\overline{3}}^s_{p^2})^q_I,
~(1_{p^0}^y\overline{\overline{6}}_{p^2}^y4_{p^0}^y\overline{5}_{p}^y{2}_{p^0}^y)^q(\overline{\overline{3}}^s_{p^2})^q_I.$$
If we insert $6^{c_i}$ right after the element $4$, we get the following permutations:
$$(1_{p^0}^x4_{p^0}^x6_{p^0}^y\overline{5}_{p}^y{2}_{p^0}^y)^q(\overline{\overline{3}}^s_{p^2})^q_I,~
(1_{p^0}^x4_{p^0}^y\overline{6}_{p}^y\overline{5}_{p}^y{2}_{p^0}^y)^q(\overline{\overline{3}}^s_{p^2})^q_I,
~(1_{p^0}^x4_{p^0}^y\overline{\overline{6}}_{p^2}^y\overline{5}_{p}^y{2}_{p^0}^y)^q(\overline{\overline{3}}^s_{p^2})^q_I.$$
\end{example}

Along the same lines as in the proof of Lemma~\ref{derangment-grammar001}, it is routine to check the eighth main result of this paper and we omit the proof of it for simplicity.
\begin{lemma}\label{derangment-grammar01G7}
If $G_{14}=\{I\rightarrow qI\left(t+sp[r-1]_p\right),~t\rightarrow xy+p[r-1]_py^2,~s\rightarrow xy+p[r-1]_py^2,~x\rightarrow xy+p[r-1]_py^2,~y\rightarrow xy+p[r-1]_py^2\}$,
then
\begin{equation*}\label{derangment-grammar022}
D_{G_{14}}^n(I)=I\sum_{\pi^c\in \z_r \wr \msn}x^{\exc_A(\pi^c)}y^{\aexc_A(\pi^c)}s^{\single(\pi^c)}t^{\fix(\pi^c)}p^{\csum(\pi^c)}q^{\cyc(\pi^c)}.
\end{equation*}
\end{lemma}

\begin{theorem}\label{thm004}
One has
\begin{equation*}\label{Bnxqy03}
A_{n}^{(r)}(x,y,s,t,p,q)=\sum_{\pi\in\msn}(x+p[r-1]_py)^{\exc(\pi)}([r]_py)^{\drop(\pi)}(t+sp[r-1]_p)^{\fix(\pi)}q^{\cyc(\sigma)}.
\end{equation*}
Equivalently,
\begin{equation}\label{Bnxqy04}
A_{n}^{(r)}(x,y,s,t,p,q)=[r]_p^ny^nA_n\left(\frac{x+p[r-1]_py}{[r]_py},\frac{t+sp[r-1]_p}{[r]_py},q\right).
\end{equation}
\end{theorem}
\begin{proof}
Let $G_{14}$ be the grammar given in Lemma~\ref{derangment-grammar01G7}.
Consider a change of the grammar $G_{11}$. Setting $u=t+sp[r-1]_p,v=x+p[r-1]_py$ and $w=[r]_py$, we get
$$D_{G_{14}}(I)=qIu,~D_{G_{11}}(u)=vw,~D_{G_{11}}(v)=vw,~D_{G_{11}}(w)=vw.$$
Let $G_{15}=\{I\rightarrow qIu, u\rightarrow vw, v\rightarrow vw, w\rightarrow vw\}$.
It follows from Lemma~\ref{lemma001exc} that
\begin{equation}\label{DG8}
D_{G_{15}}^n(I)=I\sum_{\pi\in\msn}v^{\exc(\pi)}w^{\drop(\pi)}u^{\fix(\pi)}q^{\cyc(\pi)}.
\end{equation}
Then upon taking $u=t+sp[r-1]_p,v=x+p[r-1]_py$ and $w=[r]_py$ in~\eqref{DG8}, we immediately get the desired result. This completes the proof.
\end{proof}
%It follows from~\eqref{excpi} that $\exc(\pi^c)=\exc_A(\pi^c)+\fix(\pi^c)$. Hence
%\begin{equation}
%\sum_{\pi^c\in\z_r \wr \msn}x^{\exc(\pi^c)}q^{\cyc(\pi^c)}=A_{n,r}(x,1,1,x,1,1)=r^nA_n\left(x,\frac{1+(r-1)x}{r},q\right).
%\end{equation}
%Thus
%$$\sum_{\pi^c\in\z_r \wr \msn}x^{\exc(\pi^c)}q^{\cyc(\pi^c)}=\sum_{\pi\in\msn}(x+r-1)^{\we(\pi)}r^{n-\we(\pi)}q^{\cyc(\pi)}.$$
Note that
$$A_{n}^{(r)}(x^r,1,1,1,x,q)=\sum_{\pi^c\in\z_r \wr \msn}x^{\fexc(\pi^c)}q^{\cyc(\pi^c)},$$
$$A_{n}^{(r)}(x^r,1,s,0,x,q)=\sum_{\pi^c\in \D_{n,r}}x^{\fexc(\pi^c)}s^{\single(\pi^c)}q^{\cyc(\pi^c)}.$$
By using~\eqref{Bnxqy04}, we get
\begin{equation}\label{Anrxq01}
A_{n}^{(r)}(x^r,1,1,1,x,q)=[r]_x^nA_n(x,1,q),
\end{equation}
\begin{equation}\label{Anrxq}
A_{n}^{(r)}(x^r,1,s,0,x,q)=[r]_x^nA_n\left(x,\frac{sx[r-1]_x}{[r]_x},q\right).
\end{equation}
It is well known that if $f(x)$ and $g(x)$ are both symmetric unimodal polynomials with nonnegative coefficients, then so is $f(x)g(x)$.
Therefore, combining~\eqref{Anrxq} and Proposition~\ref{Zeng12},
we get the following result, which is a generalization of~\cite[eq.~(2.5)]{Zeng16}.
\begin{corollary}
For $n\geqslant 1$, one has
$$\sum_{\pi^c\in \D_{n,r}}x^{\fexc(\pi^c)}s^{\single(\pi^c)}q^{\cyc(\pi^c)}=\sum_{i=0}^n\binom{n}{i}(qsx[r-1]_x)^i[r]_x^{n-i}d_{n-i}(x,q),$$
and so the $\single$ and $\cyc$ $(s,q)$-flag derangement polynomials are symmetric unimodal when $q>0$ and $s>0$.
In particular, one has
\begin{equation}\label{Dn2}
\sum_{\pi^c\in \D_{n,2}}x^{\fexc(\pi^c)}s^{\single(\pi^c)}q^{\cyc(\pi^c)}=\sum_{i=0}^n\binom{n}{i}(qsx)^i(1+x)^{n-i}d_{n-i}(x,q).
\end{equation}
Furthermore,
$$\sum_{\pi^c\in \D_{n,2}}x^{\fexc(\pi^c)}s^{\single(\pi^c)}q^{\cyc(\pi^c)}=\sum_{k=1}^{n}\left(\sum_{i+j=k}\binom{n}{i}(qs)^i\sum_{\pi\in \md_{n-i,j}}q^{\cyc(\pi)}\right)x^{k}(1+x)^{2n-2k}.$$
\end{corollary}

From~\eqref{Dn2}, we see that
$$\sum_{\pi^c\in \D_{n,2}}x^{\fexc(\pi^c)}(1+x)^{\single(\pi^c)}=(1+x)^n\sum_{\pi\in\msn}x^{\we(\pi)}=x(1+x)^nA_n(x).$$
It is well known that (see~\cite{Bagno04,Brenti00,Zeng02} for details)
\begin{equation}\label{Anx11}
A_n(x,1,-1)=-(x-1)^{n-1},
\end{equation}
\begin{equation}\label{Anx12}
A_n(x,0,-1)=-x[n-1]_x.
\end{equation}
Combining~\eqref{Anrxq01} and~\eqref{Anx11}, one can easily find that
$$\sum_{\pi^c\in\z_r \wr \msn}x^{\fexc(\pi^c)}(-1)^{\cyc(\pi^c)}=-\frac{(x^r-1)^n}{x-1},$$
which was obtained by Bagno and Garber~\cite[Theorem~1.1]{Bagno04}.
Combining~\eqref{Anrxq} and~\eqref{Anx12}, we get
$$A_{n}^{(r)}(x^r,1,0,0,x,-1)=\sum_{\substack{\pi^c\in \D_{n,r}\\ \single(\pi^c)=0}}x^{\fexc(\pi^c)}q^{\cyc(\pi^c)}=-x[n-1]_x[r]_x^n.$$
%%%%%%%%%%%%%%%%%%%%%%%%%%%%%%%%%%%%%%%%%%%%
\section{Concluding remarks}
%%%%%%%%%%%%%%%%%%%%%%%%%%%%%%%%%%%%%%%%%%%%
%%%%%%%%%%%%%%%%%%%%%%%%%%%%%%%%%%%%%%%%%%%%%%%%%%%%%%%%%%%%%%%%%%%%%%%%%%%%%%%%%%
%%%%%%%%%%%%%%%%%%%%%%%%%%%%%%%%%%%%%%%%%%%%%%%%%%%%%%%%%%%%%%%%%%%%%%%%%%%%%%%%%%
%%%%%%%%%%%%%%%%%%%%%%%%%%%%%%%%%%%%%%%%%%%%%%%%%%
This paper gives a systematic study of excedance-type polynomials, and
a sufficient condition for a polynomial to be alternatingly increasing is also established.
%\begin{equation*}
%\sum_{\pi\in\mc_n}2^{n-\fix(\pi)}x^{\exc(\pi)}=\sum_{\pi\in\msn}(4x)^{\lpk(\pi)}.
%\end{equation*}
As pointed out by Athanasiadis~\cite{Athanasiadis17,Athanasiadis20}, it is still a challenging problem to prove
the symmetric decompositions of some combinatorial polynomials by using group actions.
It would be interesting to give a proof of~\eqref{Ankx-decom} or Theorem~\ref{dnrx-gamma} by introducing some modified Foata-Strehl group actions.

\end{document}